\documentclass[12pt,leqno]{amsart}
\usepackage{amsmath}
\usepackage{amssymb}
\usepackage{amsthm} 
\usepackage{epsf}
\usepackage{graphicx}
\usepackage{esint} 
\usepackage{dsfont}
\usepackage{hyperref} 
\hypersetup{backref,  pdfpagemode=FullScreen,  colorlinks=true} 
\usepackage[french,english]{babel}

\numberwithin{equation}{section}

\newtheorem{lem}{Lemma}[section]

\newtheorem{defi}[lem]{Definition}
\newtheorem{thm}[lem]{Theorem}
\newtheorem{cor}[lem]{Corollary}
\theoremstyle{remark}

\newcommand{\nn}{\nonumber}

\newcommand{\R}{\mathbb{R}}
\renewcommand{\H}{\mathcal H}

\newcommand{\1}{\bf 1}

\renewcommand{\d}{\partial}
\newcommand{\dist}{\,\mathrm{dist}}

\newcommand{\sm}{\setminus}

\newcommand{\wt}{\widetilde}

\newcommand{\ol}{\overline}

\newcommand{\ntlim}{\, \mathrm{n.t.lim}}

\newcommand{\cD}{{\mathcal D}}

\newcommand{\cF}{{\mathcal F}}

\usepackage{xcolor}

\begin{document}

\title[Square Functions in Higher Co-Dim]{Square functions, non-tangential limits and harmonic measure in co-dimensions larger than one}
\author{Guy David}
\author{Max Engelstein}
\author{Svitlana Mayboroda}
  \thanks{
 G.\, David was partially supported by the Simons Collaborations in MPS grant 601941, GD,
 the ANR, programme blanc GEOMETRYA ANR-12-BS01-0014, and the European Community 
 Marie Curie grant MANET 607643 and H2020 grant GHAIA 777822.
M.\ Engelstein was partially supported by an NSF postdoctoral fellowship, NSF DMS 1703306 and by David Jerison's grant DMS 1500771. 
S.\ Mayboroda was partially supported by 
an NSF INSPIRE grant DMS-1344235, NSF DMS 1839077, a Simons Foundation Fellowship, and the Simons Collaborations in MPS grant 563916, SM.
This project is based upon work supported by the National Science Foundation under Grant No. DMS-1440140 while the authors were in residence at the Mathematical Sciences Research Institute in Berkeley, California, during the Spring 2017 semester. The authors also want to thank Camille Labourie for pointing out the reference \cite{Bor}. Finally, the authors want to thank several anonymous referees, whose careful readings and comments greatly improved this paper.}

\subjclass[2010]{Primary: 28A75, 35J70. Secondary: 42B20, 42B37, 35R35}
\keywords{Free boundary problems, uniform rectifiability, degenerate elliptic equations, 
 harmonic measure in co-dimension larger than $1$} 
\address{Univ Paris-Sud, Laboratoire de Math\'{e}matiques, 
UMR 8658 Orsay, F-91405}
\email{Guy.David@math.u-psud.fr}
\address{School of Mathematics, University of Minnesota,University of Minnesota, Minneapolis, MN, 55455, USA.}
\email{mengelst@umn.edu}
\address{School of Mathematics, University of Minnesota,University of Minnesota, Minneapolis, MN, 55455, USA.}
\email{svitlana@math.umn.edu}
\date{}
\maketitle

\begin{abstract}
In this paper, we characterize the rectifiability (both uniform and not) of an Ahlfors regular set, $E$, of arbitrary co-dimension by the behavior of a regularized distance function in the complement of that set. In particular, we establish a certain version of the Riesz transform characterization of rectifiability for lower-dimensional sets. We also uncover a special situation in which the regularized distance is itself a solution to a degenerate elliptic operator in the complement of $E$. This allows us to precisely compute the harmonic measure of those sets associated to this degenerate operator and prove that, in a sharp contrast with the usual setting of co-dimension one, a converse to the Dahlberg's theorem (see \cite{Da} and \cite{DFM2}) must be false on lower dimensional boundaries without additional assumptions. 
\end{abstract}

\tableofcontents

\section{Introduction}\label{s:Introduction}

The beginning of the XXI century has brought a series of long sought-after results enlightening connections between the scale-invariant geometric, analytic, and PDE properties of sets. Among the most celebrated ones were the Riesz transform characterizations of uniform rectifiability \cite{DS2, To1, NToV} and full description of the sets for which the harmonic measure is absolutely continuous with respect to the Lebesgue measure, in terms of uniform rectifiability along with a certain topological condition \cite{AHMMT}. Both of these results rest on, and have been surrounded by, a plethora of important related advancements. We do not intend to review the related literature in the present work but point out that virtually the entire theory has been restricted to the $n-1$ dimensional boundaries of domains in $\R^n$. The question of a possible extension of these results to lower-dimensional sets has become one the central open problems in the subject ever since.

In \cite{DFM3}, the first and third author (together with Joseph Feneuil) introduced a regularized distance function $D_{\mu,\alpha}$ (see \eqref{1.1} and \eqref{1.3} below) as a tool in a long program to characterize uniformly rectifiable sets of co-dimension greater than one by the behavior of certain (degenerate) elliptic operators in the complement of that set. In this paper we provide necessary and sufficient conditions for both the rectifiability and uniform rectifiability (see Definition \ref{def:UR}) of a $d$-Ahlfors regular set, $E$, in terms of the oscillation of $|\nabla D_{\mu,\alpha}|$ in the complement of $E$. These results are new even in the classical context of co-dimension one. However, most notably, for lower dimensional sets, and for special values of involved parameters, they provide an unexpected version of the Riesz transform characterization - we will discuss the details after appropriate definitions. 

We also discover a surprising situation in which the distance function itself is a solution to the degenerate elliptic operator introduced by \cite{DFM3}. This allows us to compute the Green function {\it explicitly} and compare the associated harmonic measure (see \eqref{e:Lharmonicmeasure} below) to the Hausdorff measure {\em no matter how irregular $E$ is} -- a situation unheard of in co-dimension one. In particular, as we mentioned above, recently, as a culmination of a long line of research starting with the work of F. and M. Riesz \cite{RR}, the results of Azzam, Hofmann, Martell, Mourgoglou and Tolsa \cite{AHMMT} showed that the harmonic measure supported on a co-dimension one set is nice if and only if the set itself is nice. More precisely, assuming a ``quantitative openness" condition on $\Omega \subset \mathbb R^n$ and condition \eqref{e:AR} on $\partial \Omega$ (for $d= n-1$), the harmonic measure of $\Omega$ supported on $\partial \Omega$ is regular if and only if $\partial \Omega$ is uniformly rectifiable and a weak connectivity condition holds inside of $\Omega$. Our result shows that the analogous characterization fails brutally in the co-dimension greater than one situation described above (see below for further discussion). 

In order to more precisely state our results (and the analogous work in co-dimension one), let us introduce some notation and notions.  We are given an Ahlfors regular set $E \subset \R^n$, with $n\geq 2$, of 
any dimension $d < n$, and an Ahlfors regular measure $\mu$ on $E$. Recall, a measure $\mu$ is $d$-Ahlfors regular if \begin{equation}\label{e:AR}
C^{-1}r^d \leq \mu(B(Q,r)) \leq Cr^d\;\; \forall Q\in E,\; \forall 0 < r < \mathrm{diam}(E).
\end{equation}

A set, $E$, is $d$-Ahlfors regular if $\mathcal H^d|_E$ is a $d$-Ahlfors regular measure, or equivalently if it supports some $d$-Ahlfors regular measure. 

The most salient class of regularity for us is uniform rectifiability. Recall that a set $E\subset \R^n$ is $d$-rectifiable (with $d \in \mathbb N$ necessarily) if 

\begin{equation}\label{e:rectifiability}
\begin{aligned}
&\text{ there 
exist countably many Lipschitz functions}\\ &f_i:\mathbb R^d \rightarrow \R^n\:\:\:\: \text{such that}\:\: \mathcal H^d\left(E\backslash 
\bigcup_i f_i(\R^d)\right) 
= 0.
\end{aligned}
\end{equation}

Uniform rectifiability is a quantitative version of this (cf. \cite{DS1}):
 
 \begin{defi}\label{def:UR}
 Let  $d$ be an integer and 
$E \subset \mathbb R^n$ be a $d$-Ahlfors regular set. We say that $E$ is $d$-uniformly rectifiable 
(with constants $ \theta > 0, L > 0$), if it has big pieces of Lipschitz images. 
That is, if for any $x\in E, \mathrm{diam}(E) >  r > 0$ there exists an $f: \R^d \rightarrow \R^n$ which is $L$-Lipschitz 
and such that 
$$
\mathcal H^{d}(E \cap B(x,r) \cap f(\R^d \cap B(0,r))) 
\geq \theta  r^d. 
$$  
 \end{defi}

The definition above is just one of several (equivalent) definitions of uniform rectifiability, which we chose because we like its geometric flavor. Through the paper we will introduce as needed (and use) other characterizations of uniform rectifiability (e.g. using $\beta$ or $\alpha$ numbers). One of the goals of this paper is to provide another characterization of uniform rectifiability using the regularized distance to $E$. Note, the sets considered in this paper satisfying \eqref{e:AR} or Definition \ref{def:UR} will always be unbounded. This choice simplifies many of our theorems and proofs and comports with the prior work in \cite{DFM2}, \cite{DFM3}.

\subsection{Regularized Distances and (Uniform) Rectifiability} 
As above let $E$ be a $d$-Ahlfors regular set and $\mu$ a $d$-Ahlfors regular measure whose support is $E$. In some cases, $\mu$ will be the restriction to $E$ of the 
$d$-dimensional Hausdorff measure $\H^d$, but not always. 

For each $\alpha > 0$ 
(often fixed in the argument), we define first a function, 
$R \equiv  R_{\mu, \alpha}$, on $\Omega = \R^n \sm E$ by 
\begin{equation}\label{1.1}
R(x) =  R_{\mu,\alpha}(x) = \int_{y\in E} |x-y|^{-d-\alpha} d\mu(y),
\end{equation}
where the convergence comes from the Ahlfors regularity of $\mu$. In fact, a simple estimate with dyadic 
annuli shows that 
\begin{equation}\label{1.2}
C^{-1} \delta(x)^{-\alpha} \leq R(x) \leq C \delta(x)^{-\alpha} \ \text{ for } x\in \R^n\backslash E,
\end{equation}
where we set $\delta(x) = \dist(x,E)$ and $C$ depends on $\mu$ and $\alpha$. 
After this, we define $D =  D_{\mu, \alpha}$ (suppressing the dependence on $\mu, \alpha$ when it is clear from context or unimportant) by
\begin{equation}\label{1.3}
D(x) \equiv D_{\mu,\alpha}(x) = R_{\mu, \alpha}^{-1/\alpha}(x) \ \text{ for } x\in \R^n\backslash E.
\end{equation}

These distances were first introduced by \cite{DFM3} to study degenerate elliptic PDE, but in the first part of our paper we are more concerned with the analytic properties of $\nabla D$, which we think of as analogous to the Riesz transform  (though it is regularized by the presence of $\alpha$ in the kernel). Indeed, 
\begin{multline}\label{nabD}\nabla D_{\mu, \alpha}(x)=-\frac{1}\alpha \left( \int_{y\in E} |x-y|^{-d-\alpha} d\mu(y)\right)^{-\frac 1\alpha-1}  \int_{y\in E} \nabla_x(|x-y|^{-d-\alpha}) \, d\mu(y)\\
=\frac{d+\alpha}\alpha \left( \int_{y\in E} |x-y|^{-d-\alpha} d\mu(y)\right)^{-\frac 1\alpha-1}  \int_{y\in E} \frac{x-y}{|x-y|^{d+\alpha+2}} \, d\mu(y), 
\end{multline}
for every $x\in \Omega$. Setting formally $\alpha=-1$ above and properly re-interpreting the integrals would transform the latter term into the classical Riesz transform. However, our $\alpha$ is always a positive number, so that the resultant expression, while analogous, is actually a quite surprising extension of the concept of the Riesz transform. One of the main discoveries of this paper is that $\nabla D_{\mu, \alpha}$ carries rich geometric information, similar to the original Riesz tranform, for sets of {\it arbitrary} dimension (not necessarily $n-1$).

To measure the oscillation of $\nabla D_{\mu, \alpha}$ in a scale invariant way we define 
\begin{equation}\label{1.4}
F(x) \equiv F_{\mu, \alpha}(x) = \delta(x) \big|\nabla (|\nabla D|^2)(x)\big| = 
\delta(x) \Big( \sum_{k=1}^n \Big|\frac{\d }{\d x_k}(|\nabla D|^2)(x)  \Big|^2 \Big)^{1/2}
\end{equation}
for $x\in \Omega \equiv \R^n \backslash E$. 
This is a dimensionless quantity, or rather, in crude terms,  it is easy to see that $D$ is Lipschitz and 
$|\nabla (|\nabla D|^2)(x)|$ is bounded by $\delta(x)^{-1}$.
We say that $D$ satisfies the ``usual square function estimates" (USFE for short) when 
 \begin{equation}\label{1.5}
F(x)^2 \delta(x)^{-n+d}dx \ \text{ is a Carleson measure on $\Omega$.}
\end{equation}

Let us recall the definition of a Carleson measure (which is intimately linked to uniform rectifiability and will be used several times),
first on $E\times \R_+$ (the standard case) and then 
on $\Omega$ (as needed above).  

\begin{defi}\label{d:CarlesonMeasure}
Let $E\subset \mathbb R^n$ be $d$-Ahlfors regular. We say that $\nu(x,r)$ is a Carleson measure on $E\times \R_+$ if there exists a $C > 0$ such that for every $X \in E$ and $R > 0$ we have that \begin{equation}\label{e:carlesonmeasure1} \nu(B(X,R)\cap E \times [0, R]) \leq CR^d.\end{equation}

Similarly, $\mathcal G$ is a Carleson subset of $E \times (0, +\infty)$ if the measure ${\1}_{{\mathcal G}}(x,r) \frac{d\H^d(x) dr}{r}$ is a Carleson measure on $E\times \R_+$. 
\end{defi}

\begin{defi}\label{d:CarlesonMeasurecompliment} 
Let $E\subset \mathbb R^n$ be $d$-Ahlfors regular. We say that $\lambda(x)$ is a Carleson measure on $\Omega \equiv \mathbb R^n \backslash E$ if there exists a $C \geq 0$ such that $$\lambda(\Omega \cap B(X,R)) \leq C R^d$$ for all $X \in E$ and $R > 0$. 

Similarly, $Z \subset \Omega$ is a Carleson set if ${\1}_{{Z}}(x) \delta(x)^{-n+d}dx$ 
is a Carleson measure on $\Omega$.
\end{defi}

Notice that for the moment we use the definition of $F$ that is the simplest  for us to use; 
taking $F(x) = \big|\nabla (|\nabla D|)(x)\big|$ instead will give the same results; see e.g. Corollary \ref{c:URimpliestildeUSFE}.

The results of Sections \ref{s:URimpliesUSFE} and \ref{s:wUSFEimpliesUR} characterize uniform rectifiability through the USFE; in particular we show 
\begin{thm}\label{tintro1} Let $n \geq 2$ and $0 < d < n$ ({\it a priori} $d$ is not necessarily an integer).  A $d$-Ahlfors regular set in $\R^n$ equipped with a $d$-Ahlfors regular measure $\mu$ is uniformly rectifiable if and only if $F_{\mu, \alpha}$ satisfies the USFE for some $\alpha > 0$.
\end{thm}
In fact a slightly weaker condition than the USFE, which we call ``weak USFE", implies that $d$ is integer and $E$ is uniformly rectifiable. We note that this characterization is new even in co-dimension one and in some instances may be easier to check than previous conditions involving square functions 
(due to David and Semmes \cite{DS1}) and singular integrals 
\cite{DS2,NToV, To1}. Indeed, while it is often hard to check 
the $L^2$ boundedness of singular operators,   given a set $E$,  $F_{\mu, \alpha}$ can be computed fairly explicitly.  We also remark the parallel with the ``classical" USFE, involving a Carleson measure condition similar to the above for the second derivative of the Newtonian kernel, that is, the gradient of the kernel of the classical Riesz transform \cite{DS2}. Just as  $\nabla D$ in \eqref{nabD} resembles the Riesz transform only formally, our expressions here with $\alpha>0$ are, of course, different, both intuitively (we really think of them as derivatives of a regularized distance) and factually 
 (these are not classical singular integrals).
Most importantly, the results here apply to the lower-dimensional setting while the (obvious extension of) the classical USFE is known to fail for sets of dimension lower than $n-1$ \cite{DS2}, p.~267, and no lower-dimensional analogue of this characterization has been known thus far. 

We would like to highlight a crucial component of the proof that the USFE implies uniformly rectifiability: 
Corollary \ref{t2.2}. There we show that if $|\nabla D_{\mu, \alpha}|$ is constant, then $E$ must be a $d$-affine space and $\mu$ a constant multiple of $\mathcal H^d|_E$. 
This corollary follows from Theorem~\ref{t:constantimpliesconvex}, 
which is possibly of wider interest, and states that if the distance to a set, $E$, is a $C^{1}$ function, the set $E$ must be convex. As we were writing, we learned that this line of inquiry is related to results in convex analysis (see Section \ref{s:flatwhenDconstant} for details).

We are also interested in the existence of non-tangential limits for $|\nabla D|$. That is, the limit of $|\nabla D(x)|$ as $x$ approaches a point $Q\in E$ without getting ``too close" to $E$ (see \eqref{e:nontangentialaccessE} for the definition of a non-tangential region and Definition \ref{d:ntlimit} for the definition of non-tangential convergence). Non-tangential limits are an important concept in harmonic analysis (e.g. the classical Fatou's theorem). In Section \ref{s:NTLimits}, we prove
\begin{thm}\label{tintro2} Let $n \geq 2$ and $0 < d < n$ (with $d$ not necessarily an integer). A $d$-Ahlfors regular set in $\R^n$ equipped with a $d$-Ahlfors regular measure $\mu$ is rectifiable if and only if the non-tangential limit of $|\nabla D_{\mu, \alpha}|$ exists at $\mu$-almost every point in $E$ (for cones of every aperture).
\end{thm}
Continuing to think of $\nabla D_{\mu, \alpha}$ as a slightly smoother version of Riesz transform, these results are in the vein of Tolsa's, \cite{To1}, who shows that the existence of principle values of the Riesz transform is equivalent to rectifiability. We remark that  
the $L^2$ boundedness of the operator 
$$\frac{d+\alpha}\alpha \left( \int_{y\in E} |x-y|^{-d-\alpha} d\mu(y)\right)^{-\frac 1\alpha-1}  \int_{y\in E} \frac{x-y}{|x-y|^{d+\alpha+2}} \,f(y) \, d\mu(y), 
 $$
naturally associated to $\nabla D_{\mu, \alpha}$, is valid on all Ahlfors regular sets, using a simple domination by the Hardy-Littlewood maximal function. 
Thus, the celebrated result of \cite{NToV}  that says that for an Ahlfors regular set $E$ of co-dimension $1$, $E$ is uniformly rectifiable 
if and only if the Riesz kernel defines a bounded operator on $L^2(E, \H^{n-1})$
trivially fails in our case. Yet, much as in \cite{To1}, the existence of the limits characterizes rectifiability, albeit these are different, non-tangential limits.

 Finally, let us remark that in both Theorems \ref{tintro1} and \ref{tintro2} we assume (implicitly) that $d$ is an integer in one direction (``(uniform) rectifiability implies control on the oscillation of $|\nabla D|$") whereas the fact that $d$ is an integer is a corollary of the geometric conclusion in the other direction (``control on the oscillation of $|\nabla D|$ implies (uniform) rectifiability").

\subsection{Harmonic Measure in co-dimension one and greater}
Associated to these distances is the degenerate elliptic PDE:
\begin{equation}\label{e:lalpha} L_{\mu, \alpha} u\equiv -\mathrm{div}\left(\frac{1}{D_{\mu, \alpha}^{n-d-1}}\nabla u\right) 
= 0 
.\end{equation} 

In \cite{DFM2}, elliptic estimates and some potential theory were 
established for solutions of $L_{\mu, \alpha}$ in the complement of $E$. Most saliently for our purposes, it was shown that a maximal principle holds and that the Dirichlet problem could be solved for continuous data. Thus, for $X \in \Omega \equiv \R^n\backslash E$, we can define a harmonic measure $\omega^X \equiv \omega^X_{\mu, \alpha}$ as the measure given by the Riesz representation theorem with the property that if $u_f$ is the unique solution to $$\begin{aligned} L_{\mu, \alpha} u_f =& 0,\:\: \text{in}\: \Omega,\\
u_f =& f,\:\: \text{in}\: E,\end{aligned}$$ then \begin{equation}\label{e:Lharmonicmeasure} u_f(X) = \int_{E} f(Q) d\omega_{\mu, \alpha}^X(Q).\end{equation}

The distances, $D_{\mu, \alpha}$, were introduced by \cite{DFM3}, as a smooth replacement for $\delta(x)$; this smoothness was essential to the proof of the co-dimension greater than one analogue of Dahlberg's theorem in \cite{DFM3}. Recall that for the Laplacian in co-dimension one, Dahlberg proved that for Lipschitz domains the harmonic measure is quantitatively absolutely continuous (precisely, an $A_\infty$-weight) with respect to surface  measure; see \cite{Da}. 
In \cite{DFM3}, the authors proved that in co-dimension greater than one, $\omega_{\mu, \alpha} \in A_\infty(d\mathcal H^d|_E)$ when $E$ is a graph with small Lipschitz constant. In work in progress, \cite{DM}, the first and third author are looking to extend this result to all uniformly rectifiable sets $E$.  

As mentioned above, the analogue of this program in co-dimension one has been a question of central interest for years, in particular, because the behavior of harmonic measure supported on a set $E$ has important consequences for the solutions of the Dirichlet problem in the complement of that set (see, e.g. \cite{FKP} and for recent results in higher co-dimension, \cite{MZ}). Recently, Azzam, Hofmann, Martell, Mourgoglou and Tolsa \cite{AHMMT} have completed this program in co-dimension one for the Laplacian. To be precise, they start with an open set $\Omega \subset \mathbb R^n$ which satisfies the interior corkscrew condition (roughly this is used to rule out cusps in $\partial \Omega$ pointing into $\overline{\Omega}^c$). They further assume that $\partial \Omega$ is $(n-1)$-Ahlfors regular. In one direction, they show that an additional connectivity near the boundary assumption (known as the weak local John condition) on $\Omega$ combined with the $(n-1)$-uniform rectifiability of $\partial \Omega$ implies that the harmonic measure of $\Omega$ is quantitatively absolutely continuous with respect to $\mathcal H^{n-1}|_{\partial \Omega}$ (they use a condition known as weak-$A_\infty$, which takes into account the pole of the harmonic measure and is natural due to the potentially nasty topology of $\Omega$). The aforementioned work of \cite{DM} should be seen as a generalization of this result to higher co-dimension. 

More impressively, uniform rectifiability and the weak local John condition are necessary and sufficient. Indeed, under the same assumptions of interior corkscrews and $(n-1)$-Ahlfors regularity of the boundary as above, \cite{AHMMT} shows that if the harmonic measure in $\Omega$ is in weak $A_\infty$, then it must be that $\partial \Omega$ is uniformly rectifiable and $\Omega$ satisfies the weak local John condition (actually that weak $A_\infty$ implies uniform rectifiability was already known, c.f. \cite{HLMN}. The contribution of \cite{AHMMT} is to show that weak $A_\infty$ implies the weak local John condition on $\Omega$). For a more precise description of these results and discussion on the interplay between these assumptions, we suggest the introduction of \cite{AHMMT}. 

Our initial goal was to connect the USFE and the existence of non-tangential limits with the behavior of 
the harmonic measure, $\omega_{\mu, \alpha}$, 
with the hopes of proving a higher co-dimension version of \cite{HLMN} and therefore characterizing uniform rectifiability by the behavior of $\omega_{\mu, \alpha}$ (note that the topological conditions above, namely the interior corkscrew condition and the weak local John condition, are satisfied by $\Omega \equiv \mathbb R^n\backslash E$ whenever $E$ is $d$-Ahlfors regular with $d < n-1$). 

However, as mentioned above, such a result turns out to be completely false in some cases.
There is a specific value of $\alpha$, described in Section \ref{s:MagicAlpha}, where $D_{\mu, \alpha}$ itself is a solution of $L_{\mu, \alpha} u=0$. In this scenario, we can explicitly compute $\omega_{\mu, \alpha}$ by showing that $D_{\mu, \alpha}$ is, in fact, the Green function with pole at infinity (see Definition \ref{d:poleatinfinity} and Corollary \ref{c:magicomegafinite}).

More precisely, we show that $\omega_{\mu,\alpha}$ is proportional to 
$\sigma = \mathcal H^d|_E$ for {\it any Ahlfors regular set $E$}, including purely unrectifiable ones (Theorem \ref{t:magicomegaatinfinity}). Thus, in the case of ``magic-$\alpha$" 
the converse to \cite{DM} (hence the 
higher co-dimension generalization of \cite{HLMN}) fails in the most spectacular way possible. 
We also show that for any rectifiable set $E$, the harmonic measure for 
magic $\alpha$ is a constant multiple of $\sigma$ (see Corollary \ref{c:poissonkernelconstant}).  
This is surprising for two reasons. First, 
for the Laplacian in co-dimension one, under mild topological assumptions, the only set for which $\omega = \sigma$ is the half-space (\cite{kenigtoroAC}; here $\omega$ is the usual harmonic measure for the Laplacian). Second, there are very few situations in co-dimension one in which the Poisson kernel, $\frac{d\omega}{d\sigma}$, or the Green function can be precisely computed. Essentially only in the presence of lots of symmetry (e.g. the ball) or where there is an explicit conformal transformation from the ball (e.g. polygonal domains in $\mathbb R^2$) are the Poisson and Green kernels known. Here, for magic $\alpha$, we are able to compute the Green function with pole at infinity {\it for any Ahlfors regular set} and the Poisson kernel $\frac{d\omega_{\mu,\alpha}}{d\sigma}$ {\it for any rectifiable set} $E$. 

We now expect that the situation for magic $\alpha$ is really exceptional, and hope to make precise how this is so in future investigations.

\section{The square function estimate for uniformly rectifiable sets}\label{s:URimpliesUSFE}

In this section we prove the direct results concerning the USFE.

\begin{thm} \label{t:URimpliesUSFE}
Let $n \geq 2$, $E \subset \R^n$ be uniformly rectifiable set of dimension $d < n$ (so $d\in \mathbb N$), and let 
$\mu$ be a 
$d$-Ahlfors regular measure whose support is $E$. Then for each $\beta > 0$, 
$D_{\mu, \beta}$ satisfies the USFE. That is, if $D_{\mu, \beta}$ is as in \eqref{1.1} and \eqref{1.3}, and $F$ is defined
as in \eqref{1.4}, then \eqref{1.5} holds.
\end{thm}

See Corollary \ref{c:URimpliestildeUSFE} concerning another, roughly equivalent, function $\wt F$.
Notice that we do not assume that $d < n-1$ here, but we do not talk about the possible
relations with the operator $L$ either.

The main outline of the proof is as follows; it is clear that $F_{\nu, \beta} = 0$ when $\nu$ is a multiple of Hausdorff measure restricted to a plane (we call these measures flat; see \eqref{e:flatmeasures}, below). The key estimate, \eqref{e:keyalphaestimate}, makes this 
quantitative:  
the size of $F_{\mu, \beta}$ can be estimated by the distance between $\mu$ and a well chosen flat measure (this distance is measured by the $\alpha$ numbers, see \eqref{e:alphanumbers}). Tolsa's characterization of uniform rectifiability using $\alpha$ numbers finishes the proof. 

\begin{proof}
Fix $\beta > 0$ and $\mu$ as in the Theorem statement. From now on we will abuse notation and refer to $D_{\mu, \beta}$ as $D$ (similarly for $R$ and $F$). Occasionally, we will have to work with $D_{\nu, \beta}$ for some other measure $\nu$. Here again we will suppress the dependence on $\beta$ and just refer to $D_{\nu}$. 

 Before we can introduce the key estimate alluded to above, we must introduce the Wasserstein distances and $\alpha$-numbers.  Let us denote by $\cF = \cF_d$ the set of flat measures;
\begin{equation}\label{e:flatmeasures}
\mbox{ a flat measure 
is a positive multiple of the Lebesgue measure on an affine $d$-plane.}
\end{equation}

We are interested in Wasserstein distances, which we define as follows. 
Given two positive measures $\mu$ and $\nu$ and a ball $B(x,r)$, 
we define $\cD_{x,r}(\mu,\nu)$ by 
\begin{equation}\label{e:Wasserstein}
\cD_{x,r}(\mu,\nu) = r^{-d-1}\sup_{f \in \Lambda(x,r)} \Big|\int_{B(x,r)} f \big( d\mu - d\nu \big) \Big|,
\end{equation}
where we denote by $\Lambda(x,r)$ the set of functions $f$ that are $1$-Lipschitz on $\R^n$ and 
vanish on $\R^n \sm B(x,r)$. Notice that the normalization is such that $\cD_{x,r}(\mu,\nu) \leq C$ 
when $\mu$ and $\nu$ are Ahlfors regular, with a constant $C$ that does not depend on $x$ or $r$.
Let us note that we will not require $\mu, \nu$ to be probability measures, so it is misleading to say that $\cD_{x,r}$ is a ``distance". 

We are especially interested in the numbers
\begin{equation} \label{e:alphanumbers}
\alpha(x,r) = \inf_{\nu \in \cF} \cD_{x,r}(\mu,\nu),
\end{equation}
where $x\in \R^n$ and $r > 0$ are such that $B(x,r)$ meets $E$. These ``$\alpha$-numbers" measure the local 
Wasserstein distances from $\mu$ to flat measures. 
In the context of quantitative rectifiability, these numbers were introduced and widely used 
by X. Tolsa (see, e.g. \cite{To}), to create a theory for measures that is analogous to P. Jones' 
$\beta$-numbers. We will use the fact that, see \cite{To}, for any uniformly rectifiable set $E$ has Carleson measure estimates on the $\alpha(x,r)^2$ (see \eqref{3.23}).

We can now introduce the key estimate: fix $x \in \Omega$ and set $r_0 = \delta(x)$ and for $k \geq 0$ let $r_k \equiv 2^k r_0$. We will prove that 
 for $1\leq i \leq n$, 
\begin{equation} \label{e:keyalphaestimate}
\begin{aligned}
\big| \d_i \big(|\nabla D(x)|^2\big) \big| &= 
\big| \d_i \big(|\nabla D(x)|^2\big) - \d_i \big(|\nabla D_{\nu}(x)|^2\big) \big|
\\
&\leq C \delta(x)^{-1} \sum_{l \geq 0} 
2^{-(\beta+1) l} 
\alpha(y,2^{l+6}r_0),
\end{aligned}
\end{equation}
 for $y \in E \cap B(x,16\delta(x))$, 
and  where $\nu$ is a correctly chosen flat measure. 
The first equality comes from the easy fact (proven below)
that for any flat measure $\nu$,  $|\nabla D_{\nu}(x)|^2$
is constant. 

Before we prove \eqref{e:keyalphaestimate}, let us see how the estimate implies the final result. Theorem 1.2 in \cite{To} says (amongst other things) that when $E$ is a $d$-dimensional 
uniformly rectifiable set and $\mu$ is an Ahlfors regular measure whose support is $E$, 
\begin{equation} \label{3.23}
\alpha(x,r)^2 \, \frac{d\mu(x) dr}{r} \ \text{ is a Carleson measure on  } E \times \R_+, 
\end{equation}
which means that for $X \in E$ and $R > 0$,
\begin{equation} \label{e:tolsaalpha}
\int_{x\in E \cap B(X,R)}\int_{r=0}^R \alpha(x,r)^2 \frac{d\mu(x) dr}{r} \leq C R^d.
\end{equation}
To be precise, in \cite{To}, the estimate \eqref{e:tolsaalpha} is not written in terms of
the numbers $\alpha(x,r)^2$, but numbers $\alpha(Q)$ indexed by dyadic pseudocubes $Q \subset E$;
however the $\alpha(Q)$ and the $\alpha(x,r)$ mutually dominate each other for comparable values of $l(Q)$ and $r$, and it is a standard argument based on Fubini's theorem to go from the condition of \cite{To} to \eqref{e:tolsaalpha}.
We skip the computation because it is both easy and done in Lemma 5.9 in \cite{DFM3}.

Notice that the function $\alpha(x,r)$ depends both on $E$ and $\mu$, and \eqref{e:tolsaalpha} contains information 
both on the geometry of $E$ (the fact that it is close to a $d$-plane in most balls) and on the distribution
of $\mu$ inside $E$. In fact, Tolsa's result is already significant when $E = \R^d$ and 
$d\mu = f d\lambda$ for some function $f$ such that $C^{-1} \leq f \leq C$.

Now we claim that the Carleson measure estimate \eqref{1.5} follows from \eqref{e:tolsaalpha} and the key inequality \eqref{e:keyalphaestimate}, let us sketch this; first we use the estimate \eqref{e:keyalphaestimate} to estimate $$\int_{B(Q,R)} F^2(X)\delta(X)^{-n+d}dX \leq C\int_{B(Q,R)}\fint_{B(X, 16 \delta(X))\cap E} a(Y, 2^6\delta(X))^2d\mu(y) \delta^{-n+d}dX,$$ where the $a$ function represents the sum on the right hand side of \eqref{e:keyalphaestimate} (this notation is from Lemma 5.89 of \cite{DFM3}). Using a Whitney decomposition of $B(Q,R)$ and changing the order of integration we can dominate the above integral by a sum over dyadic subcubes of $E\cap B(Q, 16R)$.  Arguing as above or as in Lemma 5.9 in \cite{DFM3}, we estimate $$\int_{B(Q,R)} F^2(X)\delta(X)^{-n+d}dX \leq \int_0^{16R}\int_{B(Q,s)} a(x,2^6s)^2\frac{d\mu(x)ds}{s}.$$ Finally, this last integral can be dominated by the one in \eqref{e:tolsaalpha} following the computation in Lemma 5.89 of \cite{DFM3}. 


\medskip

\noindent {\bf Proving \eqref{e:keyalphaestimate}:} To  prove \eqref{e:keyalphaestimate}, recall $F = F_{\mu, \beta}$ 
from \eqref{1.4} and  $R = R_{\mu, \beta}$ 
from \eqref{1.1}; we will need to compute their derivatives. Set $h(z) = |z|^{-d-\beta}$; this is a smooth function on $\R^{n}_\ast = \R^n \sm \{ 0 \}$, 
and we denote by $\nabla^j h$ its iterated gradient (i.e. the collection of all its derivatives of order $j$).
Notice that $R$ is smooth on $\Omega$, and 
\begin{equation}\label{e:diffR}
\nabla^j R(x) = \int_E \nabla^j h(x-y) d\mu(y).
\end{equation}
Next, $D$ is smooth on $\Omega$, and
\begin{equation}\label{e:diffD}
\nabla D(x) = -  \frac{1}{\beta} 
R(x)^{-\frac{1}{\beta}-1} \nabla R(x)
\end{equation}
and 
\begin{equation}\label{e:normgrads}
\begin{aligned}
|\nabla D(x)|^2 &= \frac{1}{\beta^2} R(x)^{-\frac{2}{\beta}-2} 
|\nabla R(x)|^2
= \frac{1}{\beta^2} R(x)^{-\frac{2}{\beta}-2} \sum_j |\d_j R(x)|^2
\end{aligned}
\end{equation}
and then 
\begin{equation}\label{e:diffnormgrad}
\begin{aligned}
\d_i \big(|\nabla D(x)|^2\big) &= -\frac{1}{\beta^2} \frac{2+2\beta}{\beta}
R(x)^{-\frac{2}{\beta}-3}  \d_i R(x)\sum_j |\d_j R(x)|^2
\\&\hskip2cm
+ \frac{2}{\beta^2} 
R(x)^{-\frac{2}{\beta}-2} \sum_j \d_j R(x) \d_i\d_j R(x).
\end{aligned}
\end{equation}
The precise structure of \eqref{e:diffnormgrad} is not so important; it suffices that we can compute (and bound) the errors we get from modifying the measure $\mu$ more or less explicitly.

The computations above are simpler when $F = F_{\nu}$ for a flat $\nu$; let $\nu = \lambda \H^d_{\vert P}$ for some $\lambda > 0$. In this case, 
\begin{equation} \label{e:Rforflat}
R_{\nu}(x) = c_1 \lambda \delta_P^{-\beta}(x),
\end{equation}
where $c_1 > 0$ is a constant that depends on $d$ and $\alpha$, and $\delta_P(x) \equiv \dist(x,P)$. 
In this case \eqref{1.3} yields
\begin{equation} \label{e:Dforflat}
D_{\nu}(x) = c_2 \lambda^{-1/\beta} \delta_P(x),
\end{equation}
with $c_2 = c_1^{-1/\beta}$ 
It follows that $|\nabla D_{\nu}|^2 = c_2^2 \lambda^{-2/\beta}$, 
and $\d_i \big(|\nabla D_{\nu}(x)|^2\big) = 0$ for $1 \leq i \leq n$.

Our next goal is to estimate the differences, $|\nabla^j R(x) -\nabla^j R_{\nu}(x)|$ where $j \geq 0$ and $\nu$ is a well chosen flat measure. In turn, these will allow us to estimate the difference, $\d_i \big(|\nabla D(x)|^2\big)- \d_i \big(|\nabla D_\nu(x)|^2\big)$. Given the complexity of \eqref{e:diffnormgrad}, we expect lots of terms, but they will all involve differences of the form $|\nabla^j R(x) -\nabla^j R_{\nu}(x)|$. We will start with the simplest case, $j = 0$. 

Recall that $x\in \Omega$, $r_0 = \delta(x)$ and, for $k \geq 0$, $r_k = 2^k r_0$. 
Let $\varphi$ be a (fixed) smooth bump function such that $0 \leq \varphi \leq 1$ on $\R^n$, 
$\varphi$ is radial, $\varphi = 1$ on $B(0,8r_0)$, $\varphi = 0$ on $\R^n \sm B(0,16r_0)$; 
then let $\varphi_0 = \varphi$ and, for $k \geq 1$, $\varphi_k(x) = \varphi(2^{-k}x) - \varphi(2^{-k+1}x)$ (so that  $\varphi_k$ is supported on
$A_k \equiv \ol B(0, 2^{k+4}r_0) \sm B(0,2^{k+2}r_0)$, where $A_0 = \ol B(0,16r_0)$). 
Note that 
$\sum_{k\geq 0} \varphi_k = 1$.

Next we will choose a flat measure $\nu_k$ which is nearly optimal for the definition of 
$\alpha(x,32r_k)$. 
That is, we will let $\nu_k = \lambda_k \H^d_{\vert P_k}$ for some $\lambda_k > 0$ 
and some affine $d$-plane $P_k$, so that, in particular,
\begin{equation}\label{e:pickingnu}
\cD_{x,32r_k}(\mu,\nu_k) \leq C \alpha(x,32r_k),
\end{equation}
where $C$ depends on $n$, $d$, and the Ahlfors regularity constant for $\mu$. 
 As we shall see, at points and scales
where $\mu$ is not well approximated by flat measures, 
 rather than choosing the measure $\nu_k$ which minimizes the right hand side of \eqref{e:pickingnu}, we will prefer to make sure that we keep some control on $\lambda_k$ 
and the support of $\nu_k$. 

To pick the $\nu_k$, let $c > 0$ be small, to be chosen soon. If $\alpha(x,32r_k) \leq c$, 
let us just pick $\nu_k$ so that $\cD_{x,32r_k}(\mu,\nu_k) = \alpha(x,32r_k)$; to see that such a minimizer exists (and is nice), 
recall that $\mu(B(x,2r_k)) \geq C^{-1} r_k^{d}$, by Ahlfors regularity; if $c$ is small enough,
depending only on $n$, $d$, and the Ahlfors regularity constant for $\mu$, then any
flat measure $\eta$ such that $\cD_{x,32r_k}(\mu,\eta) \leq 2 \alpha(x,32r_k)$ must be such that
$\eta(B(x,2r_k)) \geq (2C)^{-1} r_k^{d}$ (test \eqref{e:Wasserstein} on a bump function centered at $x$).
Hence (writing $\eta = \lambda \H^d_{\vert P}$ as above)
\begin{equation}\label{e:compactnessforeta}
P \cap B(x,2r_k) \neq \emptyset \ \text{ and } \ C^{-1} \leq \lambda \leq C.
\end{equation}
It is now easy to find a minimizing $\nu_k = \lambda_k \H^d_{\vert P_k}$, where \eqref{e:compactnessforeta} holds for $P_k, \lambda_k$.  
In the remaining case where $\alpha(x,32r_k) \geq c$, we do not complicate our life, and pick
$\nu_k = \H^d_{\vert P_k}$, where $P_k$ is any $d$-plane through $B(x,2r_k)$. Then \eqref{e:pickingnu} and \eqref{e:compactnessforeta}
hold trivially.

Set $\nu = \nu_0$ and $P = P_0$. 
By the translation invariance of our problem, we may assume that
the origin lies in $P \cap B(x,2r_0)$ (this will simplify our notation, because this way 
we don't need to translate our bump functions $\varphi_k$).

Recall that $R_{\nu}(x) = c_1 \lambda_0 \delta_P^{-\beta}(x)$ by \eqref{e:Rforflat}. Recall, furthermore, the notation in \eqref{e:diffR}.  
We want to estimate the difference
\begin{equation} \label{e:DiffinR}
|R(x) - R_{\nu}(x)| 
= \Big|\sum_{k \geq 0} \int_{A_k} \varphi_k(y) h(x-y) \big(d\mu - d\nu)(y)\Big|
\end{equation}
(by \eqref{1.1}, $\sum_k \varphi_k = 1$ and $\mathrm{supp}\; \varphi_k \subset A_k$).

Notice that $\varphi_k(y) h(x-y)$ is Lipschitz in $y$, with a constant at most $C r_k^{-d-\beta-1}$,
and it vanishes outside of $B(0, 2^{k+4}r_0) \subset B(x, 2^{k+5}r_0)$; thus by \eqref{e:Wasserstein} (applied with
$C^{-1} r_k^{d+\beta+1}  \varphi_k(\cdot) h(x-\cdot) \in \Lambda(x,2^{k+5}r_0)$)
\begin{equation} \label{e:intAlessthanD}
\Big| \int_{A_k} \varphi_k(y) h(x-y) \big(d\mu - d\nu \big)(y)\Big|
\leq C r_k^{-\beta} \cD_{x,2^{k+5}r_0}(\mu,\nu).
\end{equation}
For $k=0$, $\cD_{x,2^{k+5}r_0}(\mu,\nu) = \cD_{x,32r_0}(\mu,\nu) = \cD_{x,32r_0}(\mu,\nu_0)$.
For $k \geq 1$, we use intermediate measures. We start with 
\begin{equation} \label{e:Dtriangleinequality}
\cD_{x,2^{k+5}r_0}(\mu,\nu) \leq \cD_{x,2^{k+5}r_0}(\mu,\nu_k) +  \cD_{x,2^{k+5}r_0}(\nu_k,\nu_0)
\leq \cD_{x,2^{k+5}r_0}(\mu,\nu_k) + \sum_{l = 1}^k \cD_{x,2^{k+5}r_0}(\nu_{l},\nu_{l-1}),
\end{equation}
where the triangle inequality comes directly from the definition \eqref{e:Wasserstein}. We claim that 
\begin{equation} \label{e:Dbetweenplanes}
\cD_{x,2^{k+5}r_0}(\nu_{l},\nu_{l-1}) \leq C 
 \alpha(x,2^{l+4}r_0),  
\end{equation}
because both measures $\nu_{l}$ and $\nu_{l-1}$ approximate $\mu$ well in $B(x,2^{l+4}r_0)$.
The general idea is that since the two measures are flat measures associated to planes that
pass near $x$ (i.e. \eqref{e:compactnessforeta} holds), a good control on $B(x,2^{l+5}r_0)$ implies a good control on $B(x,2^{k+5}r_0)$.
The proof is almost the same as for equation (5.83) in \cite{DFM3}, so we leave it.

We return to \eqref{e:DiffinR}: use \eqref{e:pickingnu}, \eqref{e:intAlessthanD}, \eqref{e:Dtriangleinequality}, and \eqref{e:Dbetweenplanes} to obtain 
\begin{equation} \label{e:diffinR}
|R(x) - R_{\nu}(x)| 
\leq C\sum_{k \geq 0} r_k^{-\beta} \sum_{0 \leq l \leq k} \alpha(x,2^{l+5}r_0)
\leq C \sum_{l \geq 0} r_l^{-\beta} \alpha(x,2^{l+5}r_0).
\end{equation}
 Notice that 
$\alpha(x,2^{l+5}r_0) \leq 2^{d+1} \alpha(y,2^{l+6}r_0)$ for every $y \in B(x,16r_0)$, just by \eqref{e:Wasserstein},
\eqref{e:alphanumbers}, and because $B(x,2^{l+5}r_0) \subset B(y,2^{l+6}r_0)$ and hence 
$\Lambda(x,2^{l+5}r_0) \subset \Lambda(y,2^{l+6}r_0)$ (recall, from \eqref{e:Wasserstein}, that $\Lambda(x,r)$ is the set of functions $f$ that are $1$-Lipschitz on $\R^n$ and 
vanish on $\R^n \sm B(x,r)$). Thus
\begin{equation} \label{e:DiffinR2}
\begin{aligned}
|R(x) - R_{\nu}(x)| 
&\leq C \sum_{l \geq 0} r_l^{-\beta} \alpha(y,2^{l+6}r_0)
= C r_0^{-\beta} \sum_{l \geq 0} 2^{-\beta l} \alpha(y,2^{l+6}r_0)
\\
&= C \delta(x)^{-\beta} \sum_{l \geq 0} 2^{-\beta l} \alpha(y,2^{l+6}r_0)
\end{aligned}
\end{equation}
for $y \in B(x,16r_0)$. 

This was our estimate for $R$, but we have a similar estimate for the iterated derivatives of 
$R$. That is, we start from \eqref{e:diffR} instead of \eqref{1.1}, and observe that we can compute as above,
with an extra $|x-y|^{-j}$, which transforms into an extra $r_k^{-j} \leq C \delta(x)^{-j}$ in the estimates below. 
This yields
\begin{equation} \label{e:diffinDR}
\begin{aligned}
|\nabla^j R(x) - \nabla^j R_{\nu}(x)|
&= \Big| \int_E \nabla^j h(x-y) (d\mu-d\nu)(y) \Big|
\\
&\leq  C \delta(x)^{-\beta-j} \sum_{l \geq 0} 2^{-(\beta+j) l} \alpha(y,2^{l+6}r_0)
\end{aligned}
\end{equation}
for $y \in B(x,16\delta(x))$.
Observe also that a direct estimate with \eqref{e:diffR} yields
\begin{equation} \label{e:boundongradR}
|\nabla^j R(x)| \leq C \delta(x)^{-\beta-j}. 
\end{equation}
 Let us check that
if we pick $P_0$, our initial plane, correctly we have a similar estimate for $R_{\nu,\beta}$,
 i.e., 
\begin{equation}\label{e:boundongradflat}
|\nabla^j R_{\nu}(x)| \leq C \delta(x)^{-\beta-j}.
\end{equation} Recall from the discussion 
below \eqref{e:pickingnu}, we choose $\nu_0$ such that 
$\cD_{x,32r_0}(\mu,\nu_k) = \alpha(x,32r_0)$ when $\alpha(x,32r_kr) \leq c$. 
In this regime, we claim that, perhaps by choosing $c$ a little bit smaller, the following 
 inequality holds: 
\begin{equation}\label{e:disttoEboundedbynunaught}
\dist(y,E) \leq 10^{-1} r_0 \ \text{ for } y \in P_0 \cap B(x,16r_0).
\end{equation}
Otherwise, pick $y\in P_0 \cap B(x,16r_0)$, at distance at least $10^{-1} r_0$ from $E$,
choose a Lipschitz bump function $f$, supported in $B(y,2\cdot 10^{-2} r_0)$ so that 
$f=10^{-2} r_0$ on $B(y,10^{-2} r_0)$ and $f$ is $1$-Lipschitz. 
Then \eqref{e:Wasserstein} yields
$|\int f (d\mu-d\nu_0)| \leq c r_0^{d+1}$, while $\int f d\mu = 0$ 
(because $E$ does not meet $B(y,2\cdot 10^{-2} r_0)$) and 
$\int f d\nu_0 \geq 10^{-2} r_0 \nu(B(y,\tau r_0)) \geq C^{-1}r_0^{d+1}$ 
by \eqref{e:compactnessforeta} and because $y\in P_0$. If we take $c$ small enough, 
we get a contradiction that proves \eqref{e:disttoEboundedbynunaught}. 
We deduce from this that 
\begin{equation}\label{e:xytogetherfornaught}
|y-x| \geq \frac{r_0}{2} \ \text{ for } y \in P_0,
\end{equation}
because either $y \in B(x,16r_0)$ and we use the fact that  $10^{-1}r_0 \stackrel{\eqref{e:disttoEboundedbynunaught}}{\geq} \dist(y,E) \geq \dist(x, E) - |y-x_0| = r_0 - |y-x_0|$, or else
$|y-x| \geq 16r_0$ anyway.

When $\alpha(x,32r_0) \geq c$, we decided to pick any $d$-plane through $B(x,2r_0)$,
and we simply make sure that \eqref{e:xytogetherfornaught} holds when we do this.

Once we have \eqref{e:xytogetherfornaught}, \eqref{e:boundongradflat} easily follows from \eqref{e:compactnessforeta} and the usual computations.

We may now return to our original formula, \eqref{e:diffnormgrad}. 
It says that $\d_i \big(|\nabla D(x)|^2\big)$ is a sum of $2n$ terms,
 and we claim that because of \eqref{e:boundongradR}, each of these terms 
is bounded from above by $C\delta(x)^{-1}$.

 Indeed, if we did not 
have any derivatives, we would simply get 
$C R(x)^{-\frac{2}{\beta}} = C D^{2} \leq C \delta(x)^{2}$ by \eqref{1.3} and \eqref{1.2}.
But we have three additional derivatives, which give an extra $\delta(x)^{-3}$. Altogether,
the brutal estimate is $|\d_i \big(|\nabla D(x)|^2\big)| \leq C \delta(x)^{-1}$.
By \eqref{e:boundongradflat}, we would have the same estimate when we replace $R$ with $R_{\nu,\beta}$
in some places. Now we need to estimate 
$\d_i \big(|\nabla D(x)|^2\big) - \d_i \big(|\nabla D_{\nu,\beta}(x)|^2\big)$,
which is a sum of terms like the above, except that now one of the terms of each product is
replaced with the corresponding difference involving $|\nabla^j R(x) - \nabla^j R_{\nu,\beta}(x)|$.
We use \eqref{e:diffinDR} for this difference (which allows us to multiply the estimate by a sum of $\alpha$-numbers),
keep the same estimates for the rest of each product, sum everything up, and get that
$$
\begin{aligned}
\big| \d_i \big(|\nabla D(x)|^2\big) \big| &= 
\big| \d_i \big(|\nabla D(x)|^2\big) - \d_i \big(|\nabla D_{\nu,\beta}(x)|^2\big) \big|
\\
&\leq C \delta(x)^{-1} \sum_{l \geq 0} 2^{-(\beta+1) l} \alpha(y,2^{l+6}r_0), 
\end{aligned}$$
for $y \in B(x,16\delta(x))$. This is \eqref{e:keyalphaestimate}, the desired result. 
\end{proof}

The attentive reader may ask why we raise $|\nabla D|$ to the second power in the definition of $F$ (see \eqref{1.4}). Indeed, this is done mostly for aesthetic reasons (mainly so that \eqref{e:diffnormgrad} doesn't look so nasty). In the following Corollary we show that our result still holds if $F$ is replaced by $\tilde{F}$ (which is the same except we do not square $|\nabla D|$). 

\begin{cor} \label{c:URimpliestildeUSFE}
Theorem \ref{t:URimpliesUSFE} is still valid when we replace $F(x)$ with
\begin{equation} \label{3.25}
\wt F(x) = \delta(x) \big|\nabla (|\nabla D|)(x)\big|.
\end{equation}
\end{cor}

\begin{proof}
Noting that $|\nabla D| = (|\nabla D|^2)^{1/2}$ we see that 
$\wt F(x) = \frac{1}{2} F(x) |\nabla D|^{-1}$, at least for $x$ such that $\nabla D(x) \neq 0$.
Let $C_1 \geq 0$ be large, to be chosen soon (depending on $n$, $d$, and the Ahlfors regularity 
constant for $\mu$), and set $Z = \big\{ x\in \Omega \, ; \,  |\nabla D| \leq C_1^{-1}\big\}$. 
It is enough to control $\wt F(x)  \textbf{1}_{Z}(x)$, because we can use Theorem \ref{t:URimpliesUSFE} for the rest of $\wt F$.

Even on $Z$, $\wt F$ is not as large as one may fear; for $1 \leq i \leq n$, 
\begin{equation}\label{3.26}
\begin{aligned}
\Big| \d_i\big(|\nabla D| \big)(x) \Big| &= \Big| \d_i(\sqrt{|\nabla D|^2})(x) \Big|
= \frac12 \Big| \frac{\d_i \big(|\nabla D|^2 \big)(x)}{|\nabla D(x)|} \Big|
\\
&= \Big| \frac{\d_i \nabla D(x) \cdot \nabla D(x)}{|\nabla D(x)|} \Big|
\leq \big| \nabla^2 D(x) \big| \leq C \delta(x)^{-1}
\end{aligned}
\end{equation}
by brutal computations, and at the end, \eqref{e:diffD} and \eqref{e:boundongradR}. In particular, this implies that $\wt{F}$ is bounded uniformly on $Z$:

\begin{equation}\label{e:wFboundedonZ}
|\wt{F}(x)| = \delta(x)|\nabla |\nabla D(x)|| \leq C,\:\: \forall x \in Z.
\end{equation}

In addition, we claim that $Z$ itself is not large. Indeed, let $x \in Z$ be given; recall the notation used in the proof of
Theorem \ref{t:URimpliesUSFE}, specifically that $\nu$ is a well chosen flat measure so that $D_{\nu,\beta}(z) = c_2 \lambda_0^{-1/\beta} \delta_P(z)$. 
Hence, by \eqref{e:compactnessforeta}, $|\nabla D_{\nu,\beta}(x)| \geq C^{-1}$ and, by  \eqref{e:diffD} (and \eqref{e:xytogetherfornaught}),
\begin{equation}\label{3.27}
|\nabla R_{\nu,\beta}(x)| \geq C^{-1} R_{\nu,\beta}(x)^{\frac{1}{\beta}+1} \geq C^{-1} \delta(x)^{-1-\beta}.
\end{equation}
On the other hand, $ |\nabla D| \leq C_1^{-1}$ by definition of $Z$ hence, by \eqref{e:diffD} again,
\begin{equation}\label{3.28}
|\nabla R(x)| \leq C C_1^{-1} R(x)^{\frac{1}{\beta}+1} \leq C C_1^{-1} \delta(x)^{-1-\beta}.
\end{equation}
If we choose $C_1$ large enough, we deduce from the two that 
\begin{equation}\label{3.29}
|\nabla R_{\nu,\beta}(x) - \nabla R(x)| \geq c \delta(x)^{-1 - \beta}
\end{equation}
for some $c > 0$.
Then by \eqref{e:diffinDR} (with $j=1$), 
\begin{equation}\label{3.30}
\sum_{l \geq 0} 2^{-(\beta+j) l} \alpha(y,2^{l+6}r_0) \geq C^{-1}
\ \text{ for } y \in B(x,16\delta(x)).
\end{equation}
But we have seen earlier that the work of Tolsa \cite{To} gives a Carleson estimate on the square of
the sum (over dyadic cubes) of the left hand side of \eqref{3.30} (see the discussion right before the beginning of the proof of \eqref{e:keyalphaestimate}). This Carleson estimate implies
by Chebyshev (and the same computations using Fubini that lead from \eqref{e:tolsaalpha} to \eqref{1.5}; see the two paragraphs after \eqref{e:tolsaalpha})
that $Z$ is a Carleson set. That is, there is a constant $C \geq 0$ such that for $X \in E$ and $R > 0$, 
\begin{equation}\label{e:Carlesonset}
\int_{B(X,R) \cap Z} \frac{d\mu(x)}{\delta(x)^{n-d}} \leq C R^d.
\end{equation}
This immediately leads to a Carleson bound on $\wt F|_{Z}$, 
\begin{equation}\label{3.31}
\int_{B(X,R) \cap Z} |\wt F(x)|^2 \frac{d \mu(x)}{\delta(x)^{n-d}} 
\stackrel{\eqref{e:wFboundedonZ}}{=} \int_{B(X,R) \cap Z} C^2
\frac{d\mu(x)}{\delta(x)^{n-d}}
\leq C R^d.
\end{equation}
 This completes the proof of Corollary \ref{c:URimpliestildeUSFE}.
\end{proof}

\section{$E$ is flat when $|\nabla D|$ is constant on $\Omega$.}\label{s:flatwhenDconstant}

To prove the converse to Theorem \ref{t:URimpliesUSFE} (and later in Section \ref{s:NTLimits} to study non-tangential limits), we first prove the ``limiting result": 
 if  
$F$ vanishes, i.e., if  
$|\nabla D_{\mu, \beta}|$ is constant, then $\mu$ must be supported on a plane. 
 
More precisely, we show in this section that in this case $D_{\mu, \beta}$ is a multiple of $\delta$, $E$ is a $d$-plane and
$\mu$ is a multiple of $\H^d|_E$. 

We learned while writing the paper that a subset, $C$, of a Banach space, $X$, with the property that for every $x\in X$ there is a unique closest point $c\in C$ to $x$ is called a {\bf Chebyshev set}. Chebyshev sets are well studied (see the survey, \cite{Bor}), and it is an old theorem, attributed to Bunt, that every Chebyshev set in Euclidean space is convex (Theorem 13 in \cite{Bor}). Invoking this result would allow us end the proof of Theorem \ref{t:constantimpliesconvex} after \eqref{2.3}. However, we include the whole argument for the sake of completeness. As an aside, it is apparently an interesting open question as to whether every Chebyshev set in a Hilbert space is convex. 

\begin{thm} \label{t:constantimpliesconvex}
Let $E$ be a closed set in $\R^n$, and let $D$ be a continuous nonnegative function on $\R^n$,
which vanishes on $E$, is of class $C^1$ on $\Omega = \R^n \sm E$, and such that 
$|\nabla D|$ is positive and constant on every connected component of $\Omega$.
Then $E$ is convex. If, in addition, $|\nabla D| = 1$ on $\Omega$, then $D(x) = \dist(x, E)$ for $x\in \R^n$. 
\end{thm}

Theorem \ref{t:constantimpliesconvex} is stated as is so that we may apply it easily in the proof of Corollary \ref{t2.2}. However, the discerning reader will notice that the theorem is really the combination of two separate facts: first  that a $C^1$ function vanishing on $E$ with constant derivative on a connected component of $\R^n \backslash E$ is a constant multiple of $\mathrm{dist}(x,E)$ on that component. Second, the fact, mentioned in the introduction, that if $\dist(x, E)$ is $C^1$ in $\R^n \backslash E$, then $E$ is convex.

Note that the function $D$ in Theorem \ref{t:constantimpliesconvex} is not necessarily of the form $D_{\mu, \alpha}$ defined in \eqref{1.1}. However, we will eventually apply the theorem 
to exactly those functions. 

\begin{proof}
We start with the assumption that $|\nabla D| = 1$ on some connected component $\Omega_0$ of 
$\Omega$. Observe first that $D(x) > 0$ on $\Omega_0$, because of our assumption that 
$|\nabla D| \neq 0$ (and that $D \geq 0$).
We may of course assume that $\Omega_0 \neq \emptyset$. In our main case, when $E$ is Ahlfors regular of dimension $d < n-1$, $\Omega \equiv \R^n \backslash E$ is connected and dense in $\R^n$, so $\Omega_0 = \Omega$.

Set $v(x) = \nabla D(x)$ on $\Omega_0$; this is a $C^0$ vector field that does not vanish, and we can use it to 
define a flow. That is, given $x \in \Omega_0$, we can define $\varphi(x,\cdot)$ to be the solution of the equation
$\frac{\d \varphi(x,t)}{dt} = -v(\varphi(x,t))$ such that $\varphi(x,0)=x$, which is defined on a maximal 
(open) interval $I(x)$. 
By the chain rule (and $|\nabla D| \equiv 1$), $\partial_t D(\varphi(x,t)) = -1$ for $t \in I(x)$. Integrating this,  we note that $D(\varphi(x,t)) = D(x) - t$
for $t \in I(x)$.

This solution can be extended as long as $\varphi(x,t)$ stays in $\Omega$ 
(or equivalently $\Omega_0$), which means at least as long as $D(\varphi(x,t)) > 0$. 
So $I$ contains $[0,D(x))$, and $\lim_{t \uparrow D(x)} D(\varphi(x,t)) = 0 \Rightarrow \varphi(x,D(x)) \in E$. To be precise; while the flow cannot be extended naturally to time $t= D(x)$, the limit $p(x) \equiv \lim_{t\uparrow D(x)} \varphi(x,t)$ exists and $p(x) \in E$.

Since $\delta$ and $\varphi(x,\cdot)$ are $1$-Lipschitz,
$t\mapsto \delta(\varphi(x,t))$ is $1$-Lipschitz, and 
\begin{equation}\label{2.1}
\delta(x) = \delta(\varphi(x,0)) \leq D(x) +\lim_{t\uparrow D(x)} \delta(\varphi(x,t)) = D(x).
\end{equation}

On the other hand, if $p_{\delta(x)}$ is a point of $E$ that minimizes the distance to $x$,
then the bound on the gradient of $D$ implies that $$D(x) = |D(p_{\delta(x)}) - D(x)| \leq |p_{\delta(x)}-x| = \delta(x),$$ where the first equality follows from the continuity of $D$ at $p_{\delta(x)}$.
That is, not only did we prove that
\begin{equation}\label{2.2}
\delta(x) = D(x) \ \text{ for } x\in \Omega_0,
\end{equation}
but we also learned that the flow follows straight lines. More precisely, setting \begin{equation}\label{e:defofGamma} \Gamma_x = \big\{ \varphi(x,t) ; 0 \leq t < D(x) \big\},\end{equation} we know that 
the length of $\Gamma_x$ is $D(x) = \delta(x)$, and since $|p(x)-x| \geq \delta(x)$ by definition
of $\delta(x)$, the fact that $\Gamma_x$ goes from $x$ to $p(x)$ and has a length $D(x)$
implies that it is the line segment $[x,p(x))$. 

Let us pause to point out that we have already proven the second conclusion of the theorem; 
that if $|\nabla D(x)| \equiv 1$ then $D(x) = \dist(x,E)$. To prove the first part of the theorem 
it will suffice to show that 
\begin{equation}\label{2.4}
p \ \text{ is $1$-Lipschitz on}\ \overline{\Omega}_0. 
\end{equation}
 Indeed, if we know \eqref{2.4}, 
  let us 
 assume for the sake of contradiction that $E$ is not convex. 
 In particular, there are points $a, b \in E$ be given, with $b \neq a$ such that there exists 
 an $x \in (a,b)$ (the open line segment between $a$ and $b$) that does not lie in $E$. 
 Let $\Omega_0$ denote the connected component of $\Omega$
that contains $x$, and define $p$ on $\ol\Omega_0$, as above 
(if $|\nabla D| = c  \neq 1$ on $\Omega_0$ we can always consider $D/c$ without losing generality). Denote by $I_0$ the connected component of
 $(a,b) \cap \Omega$ that contains $x$;  
this is an interval $(a',b') \subset (a,b)$, which is contained in $\Omega_0$
(because it is connected and contained in $\Omega$), and $a', b' \in E$. 
By \eqref{2.4} the length of the arc $p(I_0)$ is at most $|b'-a'|$, and since $p$ is the identity 
on $E$ (this follows from $|p(x) - x|= \delta(x)$ and $p$ continuous on $\overline{\Omega}_0$) we get that $p(I_0) = (a',b')$. In particular, $x\in p(I_0) \subset E$, 
which is a contradiction.

For the remainder of the proof we study $p$, aiming towards \eqref{2.4}. We first check that $p(x)$ is the unique closest point in $E$ to $x$. 

Observe that $\nabla D(x) = - \frac{p(x)-x}{|p(x)-x|}$ for $x\in \Omega_0$.
This is, for instance, because $\Gamma_x$ has a tangent at $x$ that points in the direction of 
$-v(x) = - \nabla D(x)$. But $\Gamma_x$ also points in the direction of $p(x)-x$, since
$\Gamma_x = [x,p(x))$. Let us deduce from this that
\begin{equation}\label{2.3}
 |y-x| > \delta(x) 
\ \text{ for } y \in E \sm \{ p(x) \},
\end{equation}
i.e., that $p(x)$ is the only point of $E$ that realizes the distance to $x$. Indeed, let $y \in E$
be such that $|y-x| = \delta(x)$, and observe that along $[y,x]$ the function $D(\xi)$
goes from $0$ to $\delta(x)$; since $D(x)$ is $1$-Lipschitz and the length of the segment is $\delta(x)$,
integrating $\partial_t D((1-t)y + tx)$ from $t= 0$ to $t= 1$ implies that  $\langle \nabla D(\xi), \frac{y-x}{|y-x|} \rangle = -1$, for all $\xi \in [x,y)$. Therefore,  $\frac{y-x}{|y-x|}$ also points in the direction of $-v(x)$. 
We conclude that $y-x$ and $p(x)-x$ are two vectors which point in the same direction 
and have the same length, hence $y=p(x)$.  
The claim, \eqref{2.3}, follows. 

Let us extend $p$ to $\overline{\Omega}_0$ by setting $p(x)=x$ when $x\in E$. We claim that $p$ is continuous
on $\overline{\Omega}_0$.
Indeed, if $\{ x_k \}$ in $\ol\Omega_0$ converges to $x$, the sequence $\{ p(x_k) \}$ is bounded
(because $|p(x_k)-x_k| = \delta(x_k)$), and it is easy to see that any point of accumulation, $y$, of
this sequence is such that $|y-x| = \lim_{k \to +\infty} |p(x_k)-x_k| = \delta(x)$, hence is equal to $p(x)$.
That is, $\{ p(x_k) \}$ converges to $p(x)$, as needed for the continuity of $p$. 

To prove the higher regularity of $p$, we start by showing that if $L_+(p(x), x)$ is the closed half line that starts from $p(x)$ 
and contains $x$, then
\begin{equation}\label{2.5}
L_+(p(x), x) \subset \Omega_0 \cup \{ p(x) \}\ \text{ and } \ 
p(y) = p(x) \ \text{ for $y \in L_+(p(x), x)$.}
\end{equation}

To prove the first part of \eqref{2.5}, first note that the half open segment $[p(x), x)$ cannot contain a point in $E$ (other than $p(x)$), otherwise that point would be closer to $x$ than $p(x)$ is. Later in this argument we will show that the rest of $L_+(p(x), x)$ also cannot contain a point in $E$, which will complete the proof that $L_+(p(x), x) \subset \Omega_0 \cup \{ p(x) \}$. 

Note that if $y \in [p(x), x]$, then $p(y) = p(x)$ is immediate by the uniqueness of $C^0$ vector flows; that is, $y = \varphi(x, t)$ for some $t \in (0, D(x)]$ and therefore $\varphi(y, s) = \varphi(x, t+s)$ for all $s \in [0, D(x)-t] = [0, D(y)]$. But as we've seen above the point where the flow starting at $y$ hits $E$ is, by definition, $p(y)$. This implies that $p(y) = \varphi(y, D(y))= \varphi(x, D(y) + t) = \varphi(x, D(x)) = p(x)$. 

 To prove \eqref{2.5} for $y  \in L_+(p(x), x) \backslash [p(x), x]$ we must reverse the flow. For $x\in \Omega_0$, we define $\varphi_+(x,\cdot)$ to be reverse flow of $\varphi$; that is  $\frac{\d \varphi_+(x,t)}{dt} = v(\varphi_+(x,t))$, with the initial value 
$\varphi_+(x,0)=x$. This function is defined on an interval $I_+ \subset [0,+\infty)$, 
and since we can check as we did for $\varphi$ above that $D(\varphi_+(x,t)) = D(x) +  t \geq D(x) > 0$ 
for $t\in I(x)$, and that we can  extend the solution as long as $\varphi_+(x,t) \in \Omega_0$, it follows that $I(x) = [0,+\infty)$. 

Let $x\in \Omega_0$  and $t_0 > 0$ be given, and set $y = \varphi_+(x,t_0)$ (note $D(y) = D(x) + t_0 > 0$ so $y\notin E$). 
Notice that $\varphi_+(x,t_0-t) = \varphi(y,t)$ for $0 \leq t \leq t_0$, 
because $\varphi_+, \varphi$ 
 come from reverse flows.   
This implies that $x \in \Gamma_{y}$ (recall the notation from \eqref{e:defofGamma}).  But we know that $\Gamma_{y}$ is a straight line from $y$ to $p(y)$. From this, $p(y) = p(x)$ immediately follows; indeed, if $p(y) \in (p(x), x]$ then $\delta(x) \leq |p(y) - x| < |p(x) - x|$ a contradiction. Similarly if $p(x) \in (p(y), y]$, thus the second part of \eqref{2.5} follows. Note we have also shown that the whole ray $L_+(p(x), x)\backslash [p(x), x]$ is contained in the image of the flow of $t\mapsto \varphi_+(x,t)$. Since $D(\varphi_+(x,t)) \geq D(x) > 0$ this image is contained in $\Omega$, which finishes the proof that $L_+(p(x), x) \subset \Omega_0 \cup \{ p(x) \}$. 

For $x\in \Omega_0$, denote by $P(x)$ the hyperplane through $p(x)$ 
which is orthogonal to $x-p(x)$. Then let $H(x)$  denote 
the half space on the other side of $P(x)$. That is, set
\begin{equation}\label{2.6}
H(x) = \big\{ z\in \R^n \, ; \, \langle z, x-p(x) \rangle \leq \langle p(x), x-p(x) \rangle \big\}.
\end{equation}
We claim that $E \subset H(x)$. To check this, we may assume that $p(x)=0$ and $x = \lambda e_n$, where $e_n$ is the last element
of the canonical basis and $\lambda > 0$. 
By the discussion above, $p(t e_n) = 0$ for every $t > 0$, and this means that 
$t = \dist(te_n,0) \leq \dist(te_n,z)$ for every $z\in E$. Write $z = ae_n + v$, with $v \perp e_n$;
then $\dist(te_n,z)^2 = |(t-a)e_n -v|^2 = (t-a)^2 + |v|^2$ and we get that 
$t^2 \leq (t-a)^2 +  |v|^2$. We let $t$ tend to $+\infty$ and get that $a \leq 0$, hence 
$\langle z, x-p(x) \rangle = \langle z, x \rangle = a \lambda \leq 0 = \langle p(x), x-p(x) \rangle$,
which means that $z\in H(x)$, as needed.

We now turn to \eqref{2.4}. Let $x, y \in \ol\Omega_0$ be given; we want to prove that 
\begin{equation}\label{2.7}
|p(x)-p(y)| \leq |x-y|.
\end{equation}
Without loss of generality, we may assume that $p(x)=0$ and $x=\lambda e_n$ for some $\lambda \geq 0$.
We have two inequalities that we can use, the fact that $p(y) \in H(x)$ (because $p(y) \in E$), which says that
\begin{equation}\label{2.8}
\langle p(y), x \rangle \leq 0,
\end{equation}
and similarly the fact that $0 = p(x) \in H(y)$, i.e., $0 = \langle p(x), y-p(y) \rangle \leq \langle p(y), y-p(y) \rangle$, or equivalently
\begin{equation}\label{2.9}
\langle p(y),y \rangle \geq |p(y)|^2.
\end{equation}
Write $y = \mu e_n + y_0$ for some $y_0 \in e_n^\perp$, and first assume that $\mu \leq 0$.
Replacing $x$ with $0$ diminishes $|x-y|$ but does not change $|p(x)-p(y)|$; thus it is enough to prove
\eqref{2.7} for $x=0$. That is, we just need to show that $|p(y)| \leq |y|$, which follows from \eqref{2.9}
and Cauchy-Schwarz.

So we may assume that $\mu > 0$. Replacing $x$ with $\mu e_n$ diminishes $|x-y|$ but does not change 
$|p(x)-p(y)|$, so as before we may assume that $\lambda = \mu$. That is, $y = \lambda e_n + y_0$.
Now write $p(y) = ae_n + by_0+z$, with $a, b \in \R$ and $z \in e_n^\perp \cap y_0^\perp$. Then
$a \leq 0$ by \eqref{2.8} and because $\lambda > 0$, \eqref{2.9} yields
\begin{equation}\label{2.10}
\lambda a + b |y_0|^2 =\langle y, p(y) \rangle \geq |p(y)|^2  = a^2 + b^2 |y_0|^2 + |z|^2.
\end{equation}
Now
\begin{equation}\label{2.11}
|p(x)-p(y)|^2 = |p(y)|^2 \leq \lambda a + b |y_0|^2
\end{equation}
and for \eqref{2.7} we just need to know that $\lambda a + b |y_0|^2 \leq |x-y|^2 = |y_0|^2$.
Since $a \leq 0$, we just need to  check that $b \leq 1$ or $y_0=0$. We return to \eqref{2.10}, which says that
\begin{equation}\label{2.11}
b(b-1)|y_0|^2 \leq \lambda a - a^2 - |z|^2.
\end{equation}
The right-hand side is non positive, so $y_0 = 0$ or else $b(b-1) \leq 0$; this 
 last case
is impossible if $b > 1$, so finally \eqref{2.7} holds and $p$ is $1$-Lipschitz.
 \end{proof}

Notice that there is a (less interesting) converse. If $E$ is convex and $D(x) = \dist(x,E)$, then $D$ is $1$-Lipschitz,
the point $p(x) \in E$ such that $|x-p(x)| = D(x)$ is unique, and it is not so hard to check that 
$\nabla D(x) = - \frac{p(x)-x}{|p(x)-x|}$ and so $|\nabla D(x)| \equiv 1$.

We now 
apply Theorem \ref{t:constantimpliesconvex} to the situation where $D = D_{\mu, \alpha}$ 
 is defined  by \eqref{1.1}. 

\begin{cor} \label{t2.2}
Let $0 < d < n$ and let $\mu$ be a $d$-dimensional Ahlfors regular measure supported on the closed set 
$E \subset \R^n$.
Suppose that for some $\alpha > 0$, the function $D_{\mu, \alpha}$ defined by 
\eqref{1.1} and \eqref{1.3} is such that on $\Omega = \R^n \sm E$, $|\nabla D_{\mu, \alpha}|$ 
is locally constant and positive. 
Then $d$ is an integer, $E$ is a $d$-plane, and the density of $\mu$ with respect to 
$\H^d_{\vert E}$ is constant. If $d< n-1$, there is a constant $c> 0$ such that 
$D_{\mu, \alpha}(x) = c \dist(x,E)$ for $x\in \Omega$.
\end{cor}

\begin{proof}
Let $\mu$, $E$, and $\alpha$ satisfy the assumptions and let $D = D_{\mu, \alpha}$. We observed earlier that $D$ is smooth
on $\Omega$, and by \eqref{1.2} it is equivalent to $\delta$ on $\Omega$, hence has a continuous extension to
$\R^n$ such that $D(x) = 0$ on $E$. Then on each connected component of $\Omega$ there is a constant $c > 0$ such that $c^{-1} D$ satisfies the assumptions of 
Theorem \ref{t:constantimpliesconvex}, hence $E$ is convex and $D$ is a constant multiple of $\delta$ on each of the connected components of $\Omega$ (at the end of this proof we show that the constant multiple must be the same on each component). 

Next we check the geometric fact that if $d < n$ and $E$ is a convex Ahlfors regular set of dimension $d$, then $d$ is an integer and 
$E$ is a subset of an affine $d$-space.

Denote by $m$ the smallest integer greater than or equal to to $d$. 
That is, $m=d$ if $d$ is an integer, and $m = [d]+1$ otherwise. 
First we check that $m=d$ (and $d$ is an integer).
Suppose that $0 \in E$. It is easy to find $m + 1$ 
 independent points in $E$, 
i.e., points $x_0, \cdots x_m \in E$ that are not contained in any $(m-1)$-plane (a 
$d$-Ahlfors regular set cannot be a subset of a $(m-1)$-dimensional plane since $m-1 < d$).

Since $E$ is convex, it contains the convex hull of the $m + 1$ points above, 
and in particular it contains an $m$-disk $\Delta$. This forces $d \geq m$, and hence $d=m$. 
In addition, let $P$ denote the affine $d$-plane that contains $\Delta$; 
notice that $E \subset P$, because otherwise $E$ contains a $(d+1)$-disk (by convexity again), and cannot be Ahlfors regular of dimension $d$.

We want to show now that $E$ is all of $P$. 
We will show a slightly more general statement, that if $d < n$ and $E$ is a convex, 
$d$-Ahlfors regular subset of a $d$-plane $P \subset \mathbb R^n$, and if 
$\delta(x) = \dist(x, E)$ is of class $C^2$ on all of $\mathbb R^n\backslash P$,
then $E = P$. Since in the present situation $D = c\delta$ and 
$D \in C^\infty(\mathbb R^n \backslash E)$, we will conclude that $E=P$.

To see this, assume  that $E \neq P$ and, without loss of generality, 
that $0$ is a boundary point of $E$, considered as a subset of $P$.
 That is, $0 \in E$ (because $E$ is closed) and every ball around $0$ contains a point in  $P\backslash E$.
Let $C$ be the set of points $e$ such that $\lambda e \in E$ for some $\lambda > 0$.
Since $E$ is convex and contains $0$,
$C$ is also the set of points $e$ such that $\lambda e \in E$ for $\lambda > 0$ small,
and then $C$ is a convex cone. Next let $e_1$ lie in the interior of $C$; such a point exists
because $E$ contains an $m$-disk $\Delta$ (as above), and $-e_1 \notin \overline{C}$ because otherwise
$0$ would be an interior point of $E$ (note if $-e_1\in \partial C$ then we can peturb $e_1$ slightly to $e_2 \in C$ such that $-e_2 \in C$, which still gives a contradiction).

 Let $e_2$ be a unit direction which is normal to $P$; we claim that $\delta$ is not $C^2$ in the direction $e_1$ at the point $e_2$.  For small $\epsilon > 0$, we have $\epsilon e_1 \in E$; this 
  is the definition of $C$.  Thus $\delta(\epsilon e_1 + e_2) = 1$ for all small enough positive $\epsilon > 0$, and consequently, $\partial_{e_1} \delta(\epsilon e_1 + e_2) = \partial^2_{e_1e_1}\delta(\epsilon e_1 +e_2) = 0$. Other the other hand, the fact that $-e_1 \notin \overline{C}$ implies that there exists some $\theta > 0$ such that $e_1 \cdot x > - (1-\theta)\|x\|$ for all $x \in E$ ($\theta$ depends on the distance between $-e_1$ and $\overline{C}$). Let $x_\epsilon$ be the closest point in $E$ to the point $e_2 - \epsilon e_1$; 
  then 
 $$\delta^2(e_2 - \epsilon e_1) \equiv \|x_\epsilon - e_2 + \epsilon e_1\|^2  
 =  1+ \epsilon^2 + \|x_\epsilon\|^2 +2\epsilon \left\langle e_1, x_\epsilon\right\rangle 
 \geq 1+ (\|x_\epsilon\| - \epsilon)^2 +2\epsilon\theta \|x_\epsilon\|.
 $$ 
 After analyzing two cases, depending on the relative size of $\|x_\epsilon\|$ and
 $\varepsilon/2$, we find that 
 $$
 \delta(e_2 - \epsilon e_1) \geq 1 + c\epsilon^2. 
 $$

Let $M_\epsilon = \sup_{t\in [0, \epsilon]} |\partial_{e_1e_1}^2\delta(e_2 - te_1)|$. 
If we assume that $\partial_{e_1} \delta$ is continuous at $e_2$ (i.e. that $\partial_{e_1}\delta(e_2) = 0$) then by the Taylor remainder theorem: $$ 1+ c\epsilon^2  \leq \delta(e_2 - \epsilon e_1) \leq 1 + M_\epsilon \epsilon^2.$$ This implies that $\lim_{\epsilon \downarrow 0} M_\epsilon > c$ which in turn implies that $\delta$ is not $C^2$ at the point $e_2$. This contradicts the initial assumption that $E \neq P$. 

We are left to prove the final claim; that $\mu$ must be a constant times $\mathcal H^d|_{P}$. Assume without losing generality that $0\in P$ is a point of density for $\mu$, 
with density $c_0 > 0$ (clearly everything is invariant under translation, but $c_0$ may depend on the point $0 \in P$).  
For $r_k \downarrow 0$ define the measure $\mu_k$ supported on $P$ by 
$\mu_k(S) \equiv \frac{\mu(r_kS)}{r_k^d}$. 
Note that $\mu_k$ is still a $d$-Ahlfors regular measure supported on $P$. 
It is then easy to see that $\mu_k \rightharpoonup c_0 \mathcal H^{d}|_{P}$ weakly as measures. By changing coordinates, $y= r_k z$, it is also clear that 
$$
D_{\mu, \alpha}(r_k x) = \left(\int_{P} \frac{d\mu(y)}{|r_kx - y|^{d+\alpha}}\right)^{-1/\alpha} 
= \left(\frac{1}{r_k^\alpha}\int_{P} \frac{d\mu(r_kz)}{r_k^d |x - z|^{d+\alpha}}\right)^{-1/\alpha} = r_k D_{\mu_k, \alpha}(x), \forall x \in \R^n\backslash E. 
$$ 
Since $D_{\mu,\alpha}(r_k x) = c\delta(r_k x) = cr_k\delta(x)$ for some $c > 0$ (which may depend on the component of $\Omega$ containing $x$) it follows that $D_{\mu_k,\alpha}(x) = c\delta(x)$. Letting $k\rightarrow \infty$ and using that $\mu_k \rightharpoonup c_0\mathcal H^d|_{P}$ we get that $D_{c_0\mathcal H^d|_P,\alpha}(x) = c\delta(x)$. However, by \eqref{e:Dforflat} for each $c > 0$ there is only one $\overline{c}> 0$ for which $D_{\overline{c}\mathcal H^d|_P,\alpha}(x) = c\delta(x)$. This implies that $c_0 = \overline{c}$, i.e. that the density of $\mu$ with respect to $\mathcal H^{d}|_P$ is the same at all points of density in $P$. In addition, $\mu$ is independent of the connected component of $\Omega$ that contains $x$ and thus the constant $c$ is the same for all connected components of $\Omega$ Therefore, $\mu = \overline{c}\mathcal H^{d}|_P$ and we are done. 
\end{proof}

Readers familiar with the concept of tangent measure will note that we essentially analyzed the tangent measures of $\mu$ at $x_0$ to obtain that $\mu$ has constant density 
with respect 
to $\mathcal H^d|_P$. This analysis was particularly easy in the case above, i.e. when $\mu$ is an Ahlfors regular measure whose support is a $d$-plane. Later, in Section \ref{s:NTLimits}, we will need to understand the behavior of $D_{\mu, \alpha}(x)$ as $x \rightarrow E$ for more complicated sets $E$. In that section we will treat the concepts of tangent measure and blowup with more care and comprehensiveness.

\section{A weak USFE implies the uniform rectifiability of $E$}
\label{s:wUSFEimpliesUR}

In this section we use ``endpoint result" of Section \ref{s:flatwhenDconstant} to prove a ({\it a priori} slightly stronger) converse to Theorem \ref{t:URimpliesUSFE}
and Corollary \ref{c:URimpliestildeUSFE}. Let us note that throughout this section $d$ is not assumed to be an integer (but will be forced to be so {\it a posteriori}). 

\begin{thm} \label{t:weakUSFEimpliesUR}
Let $n \geq 1$ be an integer, and let $0 < d < n$ be given. 
Let $\mu$ be a $d$-dimensional Ahlfors regular measure supported on the closed set $E \subset \R^n$.
Let $\alpha > 0$ be given, define $R = R_{\mu,\alpha}$, $D = D_{\mu, \alpha}$, $F = F_{\mu, \alpha}$, and $\wt F = \wt F_{\mu,\alpha}$ by \eqref{1.1}, \eqref{1.3}, \eqref{1.4},
and \eqref{3.25}. For $\varepsilon > 0$, set
\begin{equation}\label{4.1}
Z(\varepsilon) = \big\{ x\in \Omega \, ; \, F(x) > \varepsilon \big\}
\ \text{ and } \ \wt Z(\varepsilon) = \big\{ x\in \Omega \, ; \, \wt F(x) > \varepsilon \big\}.
\end{equation}
If for every $\varepsilon > 0$ $Z(\varepsilon)$ or $\wt Z(\varepsilon)$ is a Carleson set (see Definition \ref{d:CarlesonMeasurecompliment}), then $d$ is an integer
and $E$ is uniformly rectifiable.
\end{thm}

 Notice that the USFE (applied to either $F$ or $\wt F$) implies the Carleson condition on $Z(\varepsilon)$, respectively $\wt Z(\varepsilon)$, 
 in the statement, by Chebyshev; thus we will refer to the condition that 
 $Z(\epsilon)$ (or $\wt Z(\epsilon)$) is a Carleson set as the {\em weak USFE}. 

As is always the case with these types of results, what we will prove is that there is a 
constant $\varepsilon_0 > 0$, that depends on $n$, $d$, $\alpha$, and the 
Ahlfors regularity constant for $\mu$, such that if $Z(\varepsilon_0)$ 
or $\wt Z(\varepsilon_0)$ is a Carleson set, then $d$ is an integer and $E$ is uniformly rectifiable. 
But this is not such a useful difference anyway, since $\varepsilon_0$ comes from a 
compactness argument and cannot be computed. However, it does mean that one should not be concerned about the constants associated to the Carleson set $Z(\varepsilon)$ potentially blowing up as $\varepsilon \downarrow 0$. 

Before beginning the proof of Theorem \ref{t:weakUSFEimpliesUR}, let us first check that it is enough to prove the theorem for $Z$.
Indeed, recall from \eqref{1.4} and \eqref{3.25} that 
$F(x) = \delta(x) \big|\nabla (|\nabla D|^2)(x)\big|$ and 
$\wt F(x) = \delta(x) \big|\nabla (|\nabla D|)(x)\big|$.
 Thus, as observed at the beginning of the proof of Corollary \ref{c:URimpliestildeUSFE}, $F \leq 2 |\nabla D| \wt F$.
By \eqref{e:diffD}, \eqref{1.2}, and \eqref{e:boundongradR} (note that this last estimate, while presented in the context of Theorem \ref{t:URimpliesUSFE} uses only the $d$-Ahlfors regularity of $\mu$),
\begin{equation}\label{e:boundongradD}
|\nabla D|(x) \leq C R(x)^{-\frac{1}{\alpha}-1} |\nabla R(x)| \leq C \delta(x)^{1+\alpha} |\nabla R(x)| \leq C.
\end{equation}
So $F(x)\leq C \wt F(x)$. If $F(x) > \varepsilon$, then $\wt F(x)>\varepsilon/C$. That is,
 $Z(\varepsilon) \subset \wt Z(\varepsilon/C)$. If $\wt Z(\varepsilon/C)$ is a Carleson set, then
 $Z(\varepsilon)$ is a Carleson set, and if we know the result for $Z$, we can deduce the uniform rectifiability from this and get the result for $\wt Z$.

To prove the result for $Z$ we will show that the weak USFE 
(i.e. the condition that $Z(\varepsilon)$ is a Carleson set) implies 
 that $d$ is an integer and $E$ satisfies 
the condition known as the Bilateral Weak Geometric Lemma (BWGL). The BWGL property, 
along with Ahlfors regularity, characterizes uniform rectifiability and so this will complete the proof. Let us quickly recall what the BWGL is; for a more comprehensive introduction to this and other characterizations of uniform rectifiability see, e.g., \cite{DS2}.

Recall the local normalized Hausdorff distances, $d_{x,r}$ defined for $x \in \R^n$ and $r > 0$ by 
\begin{equation}\label{4.8}
d_{x,r}(E,F) = \frac{1}{r}\left( \sup\big\{ \dist(y,F) \, ; \, y\in E \cap \ol B(x,r) \big\}
+  \sup\big\{ \dist(y,E) \, ; \, y\in F \cap \ol B(x,r) \big\}\right),
\end{equation} 
where $E$, $F$ are closed sets that meet $\ol B(x,r)$ (we will not need the other cases). Using this distance for integer $d > 0$ we can define a bilateral $d$-dimensional version of P. Jones' $\beta$ numbers \cite{Jones}, which we denote $\beta_b(x,r)$, as 
\begin{equation}\label{e:bilateralbetas}
\beta_b(x,r) \equiv \inf_{P} d_{x,r}(E,P)
\end{equation}
where the infimum is taken over all affine $d$-planes $P$ that meet $\ol B(x,r)$. These numbers measure, in a two-sided way, how close the set $E$ is to being flat at the point $x$ and scale $r > 0$. We can now state the BWGL:

\begin{defi}\label{d:BWGL}
Let $E \subset \R^n$ be a closed set, $d$ be a positive integer and $\beta_b(x,r)$ be defined with respect to $E$ as in \eqref{e:bilateralbetas}. Then $E$ satisfies the condition known as the Bilateral Weak Geometric Lemma (BWGL) if the set $\mathcal R(\tau)$ defined by \begin{equation}\label{4.11}
{\mathcal R}(\tau) = \big\{ (x,r) \in E \times (0,+\infty) \, ; \,  \beta_b(x,r) \geq \tau \big\},
\end{equation}
is a Carleson subset of $E\times (0, +\infty)$ for all $\tau > 0$. Recall that any $\mathcal G \subset E\times (0, +\infty)$ is a Carleson subset of $E$ if there exists a $C > 0$ such that for all $X \in \R^n$ and $R > 0$,
\begin{equation}\label{4.12}
\int_{x\in E \cap B(X,R)} \int_{r\in (0,R]} {\1}_{{\mathcal G}}(x,r) \frac{d\H^d(x) dr}{r} \leq C R^d.
\end{equation}
\end{defi}

It is proved in \cite{DS1} that if $E$ is Ahlfors regular 
 (of some integer 
dimension $d$) and satisfies the BWGL, then it is uniformly rectifiable. 
See also Theorem I.2.4 in \cite{DS2} for the statement.
In fact, it is enough to show that ${\mathcal R}(\tau)$ is a Carleson set for a single $\tau > 0$, sufficiently small depending on the dimensions and the Ahlfors regularity constant for $E$ (see \cite{DS2}, Remark II.2.5). 

To show that the BWGL holds, we will first 
replace the $Z(\epsilon)$ with other similar sets ${\mathcal B}(\eta)$,
 which also satisfy a Carleson condition when the $Z(\epsilon)$ do, and which are more amenable 
to a later compactness argument that will invoke Corollary \ref{t2.2}.

\begin{lem} \label{t4.2}
Let $n\geq 2, d < n$ and $E\subset \mathbb R^n$ supporting a $d$-Ahlfors regular measure $\mu$.  For $M \geq 1$ (a large constant,
 to be chosen later) 
and $x\in \Omega \equiv \mathbb R^n\backslash E$, define a big (Whitney) neighborhood of $x$ as
\begin{equation}\label{4.2}
W(x) = W_M(x) = \big\{ y\in \Omega \cap B(x,M\delta(x)) \, ; \, \dist(y,E) \geq M^{-1} \delta(x) \big\}.
\end{equation}
Define the bad set, 
 ${\mathcal B}(\eta) = {\mathcal B}_M(\eta)$, by
\begin{equation}\label{e:badballs}
{\mathcal B}_M(\eta) = \big\{ x\in \Omega \, ; \, F(y) \geq \eta \text{ for some } y \in W_M(x) \big\}.
\end{equation}
With this notation, if $Z(\varepsilon)$ is a Carleson set, then 
for each large enough $M$, ${\mathcal B}_M(3\varepsilon)$ is a Carleson set as well.
\end{lem}

Thus, with our assumption that the weak USFE holds, each ${\mathcal B}_M(\eta)$ is a Carleson set.

\begin{proof}
This will be a relatively simple covering argument.  Let $\tau \in (0,1)$ be small, to be chosen soon 
(depending on $\varepsilon$). 
We define a very dense collection, $H_\tau$, of points in $\Omega$, 
 which is a maximal subset of $\Omega$ with 
the property that $|x-y| \geq \tau \max\{\delta(x), \delta(y)\}$ when $x, y \in H_\tau$ are different.

The net $H_\tau$ is useful because $F$ varies so slowly. Indeed, recalling the estimates below \eqref{e:xytogetherfornaught}, $\delta(x)^{-1} F(x) = |\nabla \big(|\nabla D(x)|^2\big)| \leq C \delta(x)^{-1}$.
The same argument, still based on the formula \eqref{e:diffnormgrad} and the estimate \eqref{e:boundongradR}, yields 
\begin{equation}\label{4.4}
|\nabla (\delta(x)^{-1} F)|(x) \leq C \delta(x)^{-2}.
\end{equation}
Let us use this to check that if $\tau$ is small enough (depending on $\varepsilon$),
\begin{equation} \label{4.10n}
|F(x) - F(x')| \leq \varepsilon
\ \text{ for $x, x' \in \Omega$ such that $|x'-x| \leq 4\tau \delta(x)$.} 
\end{equation}
First observe that $\delta(x') \geq \delta(x) - |x'-x| \geq (1-4\tau) \delta(x) \geq \delta(x)/2$
and in fact $\delta(x)/2 \leq \delta(z) \leq 2\delta(x)$ for $z\in [x,x']$.
Then, setting $G(x) = \delta(x)^{-1} F(x)$ just for the sake of the computation,
\begin{eqnarray}\label{4.5}
|F(x) - F(x')| &=& |G(x)\delta(x)-G(x')\delta(x')| \leq \delta(x)|G(x)-G(x')| + G(x')|\delta(x)-\delta(x')|
\nn\\
&\leq& C \delta(x) [|x'-x|\delta(x)^{-2}] + G(x') |x'-x|
\leq C |x'-x| \delta(x)^{-1} \leq C \tau \leq \varepsilon,
\end{eqnarray}
 where we used \eqref{4.4} and 
the fact that $G(x') =\delta(x')^{-1} F(x') \leq C \delta(x')^{-1}$ by the estimate 
above \eqref{4.4}. So \eqref{4.10n} holds.

Now let $x\in {\mathcal B}(3\varepsilon)$ be given. 
This means that we can find $y\in W_M(x)$ such that $F(y) \geq 3\varepsilon$. 

By maximality of $H_\tau$, we can find $z\in H_\tau$ such that 
$|z-y| \leq \tau \max\{\delta(z), \delta(y)\}$ (otherwise, add $y$ to $H_\tau$).
If $\tau$ is small enough, the triangle inequality yields $\delta(y) \leq 2 \delta(z)$ and 
so $|z-y| \leq 2\tau \delta(z)$.

If $\tau$ is small enough, \eqref{4.5} implies that $F(z) \geq 2 \varepsilon$. 
In fact, this stays true for all
$w\in B(z, \tau \delta(z))$. Notice also that since $y\in W_M(x)$, 
$(2M)^{-1} \delta(z) \leq \delta(x) \leq 2M \delta(z)$, and also 
$|x-z| < 2M\delta(x) \leq 4M^2 \delta(z)$. In short, $x \in V(z)$, where
$$
V(z) = \big\{ x\in \Omega \cap B(z, 4M^2 \delta(z))\, ;\, \delta(x) \geq (2M)^{-1} \delta(z) \big\}.
$$

We are ready for the Carleson estimate. Recall \eqref{e:Carlesonset} and set 
$d\sigma(x) = \delta(x)^{-n+d} dx$; we need to show that
\begin{equation}\label{4.6}
A(X,R) := \int_{x\in \Omega \cap B(X,R) \cap {\mathcal B}(3\varepsilon)} 
d\sigma(x) \leq C R^{d}
\end{equation}
for $X\in E$ and $R > 0$. 
 Let $X$  
and $R$ be given. 
Observe that if $x\in \Omega \cap B(X,R) \cap {\mathcal B}(3\varepsilon)$, 
any point $z\in H_\tau$ constructed as above lies in $H_\tau \cap B(X,3MR)$, 
because $|z-x| \leq 2M\delta(x)$ and $\delta(x) \leq |x-X| \leq R$. 
Furthermore, the argument above tells us that every 
$x\in \Omega \cap B(X,R) \cap {\mathcal B}(3\varepsilon)$ 
is in $V(z)$ for some $z \in H(\tau, X, M, R, \varepsilon) \equiv 
\{z \in H_\tau \cap B(X,3MR)\, ; \, F(z) \geq 2\varepsilon\}$. 
 Thus  
\begin{eqnarray}\label{4.7}
A(X,R) &\leq& \sum_{z\in H(\tau, X, M, R,\varepsilon)} \int_{x\in V(z)} d\sigma(x)
\leq C \sum_{z\in H(\tau, X, M, R,\varepsilon)} \delta(z)^{-n+d} |V(z)|
\nn\\
&\leq& C \sum_{z\in H(\tau, X, M, R,\varepsilon)} \delta(z)^{d}
\leq C \sum_{z\in H(\tau, X, M, R,\varepsilon)} \sigma(B(z,\tau \delta(z)/10))
\nn\\
&\leq& C \sigma (Z(\varepsilon) \cap B(X,4MR)) \leq C R^d,
\end{eqnarray}
by definition of $\sigma$ (for the second inequality) and 
the fact that $V(z) \subset B(z, 4M^2 \delta(z))$ (for the third one), then
because $\sigma(B(z, \tau\delta(z)/10)) \geq C^{-1} \delta(z)^{d}$ and the balls 
$B(z,\tau \delta(z)/10)$ are disjoint by definition of $H_\tau$, and finally (for the last line)
since each $B(z,\tau \delta(z)/10)$ is contained in $Z(\epsilon)\cap B(X, 4MR)$, and
by our Carleson estimate assumption on $Z$. The lemma follows.
\end{proof}

As mentioned above, the set $\mathcal B_M(3\varepsilon)$ 
is defined in the right way to make it amenable to a compactness argument. 
In the following lemma we will show that if $x$ is not in $\mathcal B_M(3\varepsilon)$ 
then the set $E$ is relatively flat in a neighborhood of $x$ of radius comparable to $\delta(x)$. 

\begin{lem} \label{t4.3}
For each choice of $0 < d < n$, $\alpha > 0$, an Ahlfors regularity constant $C_0$, 
and constants $\eta >0$ (small) and $N \geq 1$ (large), we can find $M \geq 1$ 
and $\varepsilon > 0$ such that if $\mu$ is 
Ahlfors regular (of dimension $d$, constant $C_0$, and support $E \subset \R^n$), and if 
$x \in \Omega \sm {\mathcal B}_M(3\varepsilon)$, 
 then $d$ is an integer and 
there is a $d$-plane $P$ such that $d_{x,N\delta(x)}(E,P) \leq \eta$.
\end{lem}
 
More explicitly, if $d$ is not an integer, we can find $M$ and $\varepsilon$ 
(depending on $d$ too) such that $\Omega \sm {\mathcal B}_M(3\varepsilon)$ is empty.

\begin{proof}
We will prove this by compactness. 
That is, let $0 < d < n$, $C_0$, $\alpha>0$, $N$, and $\eta > 0$ be given, 
and suppose that for each $k \geq 0$, there is a set, $E_k$, a $d$-Ahlfors regular measure, 
$\mu_k$, with constant $C_0$  and whose support is $E_k$, which provide a counterexample 
with $M_k = 2^k$ and $\varepsilon_k = 2^{-k}$.
That is, let $F_k$ be defined as in \eqref{1.4} but adapted to $E_k, \mu_k$ and $\alpha$. 
We assume that there are points $x_k \in \Omega_k  \equiv \R^n \backslash E_k$, 
that do not lie in the corresponding bad set ${\mathcal B}_{M_k}^{E_k}(3\varepsilon_k)$, i.e., 
\begin{equation} \label{4.14n}
\text{$F_k(y) < 2^{-k}$ for all $y \in \Omega_k \cap B(x_k, 2^k \delta_{E_k}(x_k))$ 
with $\delta_{E_k}(y) \geq 2^{-k}\delta_{E_k}(x_k)$,}
\end{equation}
and yet for which the conclusion does not hold. That is, either $d$ is not an integer,
or $d$ is an integer but there is no $d$-plane $P_k$ such that $d_{x_k,N\delta(x_k)}(E_k,P_k) \leq \eta$.
We want to reach a contradiction.

By translation and dilation invariance, we may assume that $x_k = 0$ and 
$\delta_{E_k}(x_k) = \dist(0,E_k) = 1$.
We use the uniform Ahlfors regularity to replace $\{ (E_k,\mu_k) \}$ with a subsequence 
for which $\mu_k$ converges (in the weak sense) to an Ahlfors regular measure $\mu_\infty$, and $E_k$ converges (in the Hausdorff distance sense) to a closed set $E_\infty$ 
(locally in $\R^n$). 
It is also easy to check that $E_\infty$ is the support of $\mu_\infty$ and that $\mu_\infty$ is $d$-Ahlfors regular with a constant that depends only on $C_0$ and $n$. For more details (albeit in a slightly less general context) see the discussion before Lemma \ref{l:convunderblowups} below. 

Additionally, $R_{\mu_k, \alpha} \rightarrow R_{\mu_\infty, \alpha}$ 
(uniformly on compact sets of $\Omega_\infty \equiv \R^n \backslash E_\infty$) 
and similarly for $D_{\mu_k, \alpha}$ and 
 its derivatives.  
This follows from the weak convergence of 
 the $\mu_k$
(actually a little work is necessary as $\partial^j h(x-y)$ is not compactly supported, but one can argue 
exactly as in Lemma \ref{l:convunderblowups} below). Because of this, and with hopefully obvious notation,
\begin{equation}\label{4.10}
F_\infty(y) = \lim_{k \to +\infty} F_k(y)
\end{equation}
for every $y\in \Omega_\infty$.

Let $W_k \equiv W_{2^k}(0)$ be as in \eqref{4.2} but associated to the set $E_k$. 
Clearly, any $y \in \Omega_\infty$ lies in $W_k$ for $k$ large, and so, by assumption, 
$F_k(y) \leq 2^{-k}$ for $k$ large. Taking limits, \eqref{4.10} implies that $F_\infty(y) = 0$ 
for $y\in \Omega_\infty$ and, by \eqref{1.4}, $|\nabla D_\infty|$ is locally constant.  If by bad luck $|\nabla D_\infty|=0$ on some connected component $\Omega_0\subset \Omega_\infty$, we also get that 
$D_\infty=0$ on 
$\Omega_0$ (because $D_\infty$ vanishes on $E$); this is impossible, by the 
 definition of $D_\infty$ 
(cf. \eqref{1.2} and \eqref{1.3}). 
So $|\nabla D_\infty| \neq 0$ on $\Omega_0$, and now Corollary \ref{t2.2} says that 
$d$ is an integer and $E_\infty$ is a $d$-plane.

Now recall that $E_k$ converges to the $d$-plane $E_\infty$; we thus get that for $k$ large, 
$d_{0,N}(E_k,E_\infty) \leq \eta$,  
 the desired contradiction.
Lemma \ref{t4.3} follows.
\end{proof}

We are now ready to prove Theorem \ref{t:weakUSFEimpliesUR}.

\begin{proof}[Proof of Theorem \ref{t:weakUSFEimpliesUR}] In view of Lemma \ref{t4.2}, to prove Theorem \ref{t:weakUSFEimpliesUR},
we just need to choose $\varepsilon = \varepsilon_\tau > 0$ such that 
\begin{equation}\label{4.13}
\text{if ${\mathcal B}(3\varepsilon)$ 
is a Carleson set in $\Omega$, then ${\mathcal R}(\tau)$ is a Carleson set in 
$E \times (0,+\infty)$.}
\end{equation}
Let us do this. For $(x,r) \in {\mathcal R}(\tau)$, we first use the Ahlfors regularity of $E$ 
to choose $y\in \Omega \cap B(x,r/2)$ such that $\delta(y) \geq 2 \kappa r$, 
where the constant $\kappa > 0$ 
depends on the dimensions and the Ahlfors regularity constant.
The existence of $y$ is standard; if we could not find it, we would be able to find 
 $C_n \kappa^{-n}$ 
balls $B_j$ of radius $\kappa r/2$, centered on $E \cap B(x,r/2)$, and that are disjoint. 
This would yield
$$
C_n^{-1} \kappa^{-n}
(\kappa r/2)^d
\leq C \sum_j \H^d(E \cap B_j) \leq C \H^d(E \cap B(x,r)) \leq C r^d,
$$
a contradiction for $\kappa$ small
 because $d < n$. Denote by $H(x,r)$ the ball $B(y, \kappa r)$. Then
\begin{equation}\label{4.14}
|z-x| \leq r \text{ and } \delta(z) \geq \kappa r \ \text{ for } z\in H(x,r).
\end{equation}

Take $N = 10\kappa^{-1}$; this way, $B(z,N\delta(z))$ contains $B(x,r)$ for $z\in H(x,r)$.
 Let $\eta > 0$, 
to be chosen soon in terms of $\tau, N$, and choose 
 $M = M_{N, \eta} > 0$ and 
$\varepsilon = \varepsilon_{N, \eta} > 0$ as in Lemma \ref{t4.3}. That lemma says that 
if $z\in H(x,r) \sm {\mathcal B}_M(3\varepsilon)$, 
then we can find a $d$-plane $P$ such that $d_{z,N\delta(z)}(E,P) \leq \eta$. 
This also implies that $d_{x,r}(E,P) \leq r^{-1} N\delta(z) d_{z,N\delta(z)}(E,P) 
\leq N d_{z,N\delta(z)}(E,P) \leq N\eta$
(because $B(z,N\delta(z))$ contains $B(x,r)$ and by the definition \eqref{4.8}). 
We choose $\eta$ so small that $N\eta < \tau$, and we get that $\beta_b(x,r) < \tau$. 
This contradicts the fact that $(x,r) \in {\mathcal R}(\tau)$, therefore, every $z\in H(x,r)$ 
lies in ${\mathcal B}_M(3\varepsilon)$. 

Return to the proof of \eqref{4.13} and assume that ${\mathcal B}_M(3\varepsilon)$ 
is a Carleson set in $\Omega$. Let $X\in E$ and $R > 0$ be given, and denote 
by $A(X,R)$ the left-hand side of \eqref{4.12}, with $\mathcal G = {\mathcal R}(\tau)$. 
Since $|H(x,r)| \geq C^{-1} r^n$, we see that
\begin{equation}\label{4.15}
A(X,R) \leq C \int_{x\in E \cap B(X,R)} \int_{r\in (0,R]} {\1}_{{\mathcal R}(\tau)}(x,r) r^{-n}\left(\int {\1}_{H(x,r)}(z) 
dz \right)\,\frac{d\H^d(x) dr}{r}.
\end{equation}
Of course we apply Fubini and integrate 
 in $x$ and $r$ first. 
 
Notice that $z\in B(X,2R) \cap {\mathcal B}_M(3\varepsilon)$ 
and $|x-z| \leq r \leq \kappa^{-1} \delta(z)$, so we get that
\begin{equation}\label{4.16}
A(X,R) \leq \int_{B(X,2R) \cap {\mathcal B}_M(3\varepsilon)} h(z) dz,
\end{equation}
with 
$$
h(z) = \int_{x\in E \cap B(z,\kappa^{-1} \delta(z))} 
\int_{\delta(z) \leq r \leq R} \frac{d\H^d(x) dr}{r^{n+1}}
$$
(because $r \geq |x-z| \geq \delta(z)$). The integral in $r$ is at most $C\delta(z)^{-n}$. 
Then we integrate in $x$ and get that $h(z) \leq C \delta(z)^{d-n}$. Finally,
\begin{equation}\label{4.17}
A(X,R) \leq C \int_{B(X,2R) \cap {\mathcal B}_M(3\varepsilon)} \delta(z)^{d-n} dz \leq C R^d,
\end{equation}
by the assumption that ${\mathcal B}_M(3\varepsilon)$ is a Carleson set. 
\end{proof}

\section{Blow-ups and Non-Tangential Limits of $|\nabla D_\beta|$}\label{s:NTLimits}

Throughout this section let $E\subset \mathbb R^n$ be a $d$-Ahlfors regular set with $d < n-1$ and $n \geq 2$; this assumption is not strictly necessary for all our proofs but without it we must be a bit more careful as to questions of topology and anyways it is the only scenario in which we are interested (we try to state when the result holds with $d < n$). Let $\mu$ be a $d$-Ahlfors regular measure supported on $E$. 

We are interested in the behavior of $\nabla D_{\mu, \beta}$ near $E$. One convenient tool for studying this is the blowup procedure. 

For $Q \in E$, $S \subset E,$ $y\in\Omega$,  
 and $r_i \downarrow 0$ 
we can define

\begin{equation}\label{e:blowupsforrect}
\begin{aligned}
E_{i,Q} \equiv& \, \frac{E - Q}{r_i}\\
\mu_{i,Q}(S) \equiv& \, \frac{\mu(r_iS + Q)}{r_i^d}\\
D_{i, \beta, Q}(y) \equiv& \, \frac{D_{\mu, \beta}(r_iy + Q)}{r_i}.
\end{aligned}
\end{equation}
When the point $Q$ is unimportant or clear from context we may abuse notation and refer simply to $E_i, \mu_i$ and $D_{i,\beta}$. Note that $\mu_i$ is still Ahlfors-regular (with the same constants as $\mu$) and $E_i$ is the support of $\mu_i$. To explain the definition of $D_i$,  let $y\in \mathbb R^n \backslash E_i$, which implies that $y = \frac{z- Q}{r_i}$ for some $z \in \mathbb R^n \backslash E$. Then we can calculate

\begin{equation}\label{e:dmuianddi}D_{\mu_i, \beta}(y)^{-\beta} \equiv \int_{E_i} \frac{d\mu_i(x)}{|x-y|^{d + \beta}} \stackrel{w = r_ix + Q \in E}{= }\int_{E} \frac{d\mu(w)}{r_i^d|\frac{w-Q}{r_i} - \frac{z-Q}{r_i}|^{d+\beta}} = \left(\frac{D_{\mu, \beta}(z)}{r_i}\right)^{-\beta}.\end{equation}

As the $\mu_i$ are uniformly Ahlfors regular, we know that, perhaps passing to a subsequence, 
we have $\mu_i \rightharpoonup \mu_\infty$. 
 Since the $\mu_i$ are uniformly Ahlfors regular, $\mu_\infty$ is also Ahlfors regular and 
its support, $E_\infty$, is the limit (in the Hausdorff distance sense) of the $E_i$. We want to show that 
 $R_i$ and $D_i$ converge to $R_\infty$ and $D_\infty$.

\begin{lem} \label{l:convunderblowups}
Let $E, \mu$ be as above and  $r_k \downarrow 0$ and $Q \in E$. 
With the notation  and assumptions above, $R_k$, $D_k$, and their derivatives 
converge to $R_\infty$, $D_\infty$, and their 
derivatives, uniformly on every compact subset of $\Omega_\infty = \R^n \sm E_\infty$.
\end{lem}

\begin{proof}
 Consider $\nabla^j R_k(x) = \int_{E_k} \nabla^j h(x-y) d\mu_k(y)$,
as in \eqref{e:diffR}, and fix a compact set $K \subset \Omega_\infty$. 
Also let $\varepsilon > 0$ be given.
Choose $R > 0$ large enough (depending on $K, \varepsilon$), there exists a smooth cut-off function $\varphi \equiv \varphi_{\varepsilon, R}$, supported in the large ball $B(0,R)$, 
and chosen close enough to $\chi_{B(0,R)}$ so that 
$\int |\nabla^j h(x-y)| |1-\varphi(y)| d\mu_k(y) \leq \varepsilon$ for all $k$, and for all $x\in K$. 

Let $C \subset \mathbb R^n \backslash K$ be a closed set such that $E_k \subset C$ for all large enough $K$ (such a $C$ exists by the compactness of $K$ and the Hausdorff convergence of $E_k \rightarrow E_\infty$). Then the functions $\{\nabla^j h(x-y) \varphi(y)\}_{x\in K}$, are uniformly bounded, as functions of $y$, in $\mathrm{Lip}(\overline{B(0,r)}\cap C)$ (this bound will depend on $C, R, \varepsilon$ but that is irrelevant). By Arzela-Ascoli this sequence (indexed by $x$) is precompact. Thus we can find a finite collection of continuous functions, $\{g_i\}$, supported in $B(0, 2R)$ such that for each $x\in K$ there is a $g_i$ with 
$|g_i(y)-\nabla^j h(x-y) \varphi(y)| \leq \varepsilon R^{-d}$ for all $y\in E_k\cap B(0,2R)$ (for any $k$ large enough). Then by Ahlfors-regularity $\int |g_i(y)-\nabla^j h(x-y) \varphi(y)| d\mu_k 
+ \int |g_i(y)-\nabla^j h(x-y) \varphi(y)| d\mu_\infty \leq C \varepsilon$ for all $k$. Since each
$\int g_i d\mu_k$ converges to $\int g_i d\mu_\infty$, we see that for $k$ large,
$$\Big|\int \nabla^j h(x-y)  [d\mu_k -d\mu_\infty]\Big| \leq C \varepsilon;
$$
the conclusion (for $\nabla^j R_k(x)$) follows. 
The same estimates for $\nabla^j D$ follow as well, because on the compact set $K$ 
we have uniform lower and upper bounds on the $R_k$ (again for $k$ large enough). This proves the lemma.
\end{proof}

Lemma \ref{l:convunderblowups} gives convergence on compact sets separated from $E_\infty$. But we want to understand the convergence up the boundary. In order to do this it will be convenient to introduce ``non-tangential access" regions, for reasons that we will make clear shortly.   For $Q\in E, R > 0$ and $\eta \in (0,1)$ define:

\begin{equation}\label{e:nontangentialaccessE}
\Gamma_{R,\eta}(Q) = \big\{ x\in \Omega \cap B(Q,R) \, ; \, \dist(x,E) \geq \eta |x-Q| \big\}.
\end{equation}

Associated to these non-tangential regions is the concept of a non-tangential limit. 

\begin{defi}\label{d:ntlimit}
We say that $f$ has a non-tangential limit, $L$, at $Q \in E$, 
if there is some $\eta \in (0,1)$ such that 
$$
\lim_{R\downarrow 0} \sup_{x\in \Gamma_{R, \eta}(Q)} |f(x) -L| = 0.
$$ 
We will denote this limit $L$ by $\ntlim_{x\rightarrow Q} f(x)$, or even
$\ntlim^\eta_{x\rightarrow Q} f(x)$ to be explicit.
\end{defi}

Let $E$, $Q$, and $\{r_i \}$ be as for \eqref{e:blowupsforrect}, and assume that
the $E_i$ converge to $E_\infty$. Let $\Gamma^\infty_{R, \eta}(0)$ be defined
as \eqref{e:nontangentialaccessE} but with respect to $E_\infty$. Then, after a new sequence
extraction, the sets $\frac{\Gamma_{R r_i, \eta}(Q)-Q}{r_i}$ converge
to a limit $\Gamma$, with $\Gamma^\infty_{R, \eta/2}(0) \supset \Gamma \supset 
\Gamma^\infty_{R, 2\eta}(0)$.
For the moment, we only know that $D_i$, $R_i$, and their derivatives, converge to 
$D_\infty$, $R_\infty$, and their derivatives, uniformly on compact subsets of 
$\R^n \sm E_\infty$. If we want $\nabla D_i$ to converge to $\nabla D_\infty$
uniformly on compact subsets of $\R^n$, then it should at least converge 
uniformly on each $\Gamma^\infty_{R, \eta}(0)$, which roughly corresponds,
after a change of variables, to $\nabla D$ having a non-tangential limit at $Q$ 
(in fact, for every small $\eta$).

In the following two theorems we give a characterization of the existence of non-tangential limits of $|\nabla D|$ at $\mu$-a.e. point $Q \in E$. It turns out that the existence of this limit is intimately linked to the tangent measures of $\mu$ at $Q$ (and thus the rectifiability of $\mu$). We will assume some basic familiarity with tangent measures here; for more background we suggest Chapter 17 of \cite{Mattila}.

\begin{thm}\label{thm:rectimpliesntlimits}
Let $E$ be $d$-Ahlfors regular and $d$-rectifiable (so necessarily $d\in \mathbb N$), $\mu$ be a $d$-Ahlfors regular measure supported on $E$ and $\beta > 0$ . Then for $\mu$-almost every $Q \in E$,
 the limit $\ntlim_{x\rightarrow Q}^\eta |\nabla D_{\mu, \beta}(x)|$ exists
for every $\eta > 0$. 
\end{thm}

\begin{proof}

Notice first that it will be enough to show that for \emph{each} $\eta > 0$,
the non-tangential limit $\ntlim_{x\rightarrow Q}^\eta |\nabla D_{\mu, \beta}(x)|$ exists
for $\mu$-almost every $Q \in E$, because then the exceptional set of $Q \in E$
for which the limit fails to exist for \emph{all} $\eta$ is contained in the countable union of
the exceptional sets for $\eta_i \equiv 2^{-i}$.

Let $x_i \in  \Omega \equiv \mathbb R^n\backslash E$ be a sequence of points approaching $Q \in E$ non-tangentially (i.e. 
 $x_i \in \Gamma_{R,\eta}(Q)$ for some $\eta \in (0,1), R> 0$ and $x_i \rightarrow Q$). Let $r_i = |x_i - Q|$ and define $E_i, \mu_i, D_i$ as in \eqref{e:blowupsforrect}. 

By Lemma \ref{l:convunderblowups} (perhaps passing to a subsequence) $E_i \rightarrow E_\infty$ and $\mu_i \rightharpoonup \mu_\infty$ which is a $d$-Ahlfors regular measure supported on $E_\infty$. Furthermore $D_{i} \rightarrow D_{\infty} \equiv D_{\beta, \mu_\infty}$. This convergence happens uniformly on compacta inside of $\Omega_\infty$ in the $C^\infty$ topology. Note that $X_i \equiv \frac{x_i - Q}{r_i} \in \Omega_i \cap \overline{B(0,1)}$. We also note (by the assumption that $x_i$ is a non-tangential sequence), that $\mathrm{dist}(X_i, E_i) \geq \eta$. 

Passing to a subsequence, we may assume that $X_i \rightarrow X_\infty$, and then $X_\infty \in \Omega_\infty$
because $\dist(X_\infty, E_\infty) \geq \eta$ (recall that $r_i = |x_i - Q|$).
Then by \eqref{e:dmuianddi}, 
\begin{equation}\label{e:ntlimitinblowup}
|\nabla D_\infty(X_\infty)| = \lim_{i} |\nabla D_i(X_i)| = \lim_i |\nabla D(x_i)|.
\end{equation} 

The reader may be worried because we only proved the existence of 
$\lim_i |\nabla D(x_i)|$ for a subsequence, but what will save us is that for almost
every choice of $Q \in E$, the left-hand side $L = |\nabla D_\infty(X_\infty)|$ does not depend
on $\{ x_i \}$ or the choice of subsequences. Then it will follow that all the accumulation 
points of $|\nabla D(x_i)|$, where $x_i \in \Gamma_{R,\eta}(Q)$ and $x_i$ tends to $Q$,
are equal to the number $L$ (take a sequence $\{ x_i \}$, so that $|\nabla D(x_i)|$
tends to a given accumulation point, and then proceed as above).
The existence of the non-tangential limit $\ntlim_{x\rightarrow Q}^\eta |\nabla D(x)| = L$ 
will follow.

So we look for $Q \in E$ such that $|\nabla D_\infty(X_\infty)|$ above does not depend
on $\{ x_i \}$, the choice of subsequences, or $X_\infty$ for that matter.

Since $E$ is rectifiable, $E$ has an approximate tangent $d$-plane $P'$ at almost every
point $Q\in E$ (see 15.19 in \cite{Mattila}). 
Since $E$ is Ahlfors regular, and by Exercise 41.21 in \cite{Dv}, for instance, 
$P'$ is a true tangent plane, and any limit $E_\infty$ that we get from extraction
is the vector plane $P$ parallel to $P'$. In addition, 16.5 in \cite{Mattila} says that 
(for almost every $Q\in E$), all the blow-up limits of $\sigma = \H^d_{\vert E}$ 
are flat measures, and in fact of the form $\sigma_\infty = \H^d_{\vert P}$, 
because the density of $\sigma$ is $1$ almost everywhere. In addition,
$\mu = f \sigma$ for some function $f$ such that $C^{-1} \leq f \leq C$, and
if $Q$ is a Lebesgue density point for $f$, all the blow-up limits of 
$\mu$ at $Q$ are of the form $\mu_\infty = f(Q) \lambda_P$.

Thus for almost every point $Q\in E$, we have no choice: in \eqref{e:ntlimitinblowup}, 
$|\nabla D_\infty(X_\infty)|$ must be the constant value of $|\nabla D|$ associated 
to the plane $P$ and the measure $\mu_\infty = f(Q) \lambda_P$. 
The existence of $\ntlim_{x\rightarrow Q}^\eta |\nabla D(x)|$, 
and Theorem \ref{thm:rectimpliesntlimits}, follow.
\end{proof}

What follows is the converse to Theorem \ref{thm:rectimpliesntlimits}. 
However, we note that in order to prove the rectifiability of $E$, we need 
the non-tangential limit to exist inside cones of {\it all} apertures,
as opposed to checking the existence inside cones of any given aperture.

\begin{thm}\label{thm:allntlimitsimplyrect}
Let $E$  be a set supporting the $d$-Ahlfors regular measure, $\mu$,
with $d < n$ (not necessarily an integer) and let $\beta > 0$. 
Assume that for $\mu$-almost every $Q\in E$,
the non-tangential limit $\ntlim_{x\downarrow Q}^\eta |\nabla D_{\mu, \beta}(x)|$ exists
for every aperture $\eta \in (0,1)$.
Then $d$ is an integer and $E$ is $d$-rectifiable.
\end{thm}

\begin{proof}
We will show that at $\mu$-almost every $Q\in E$, every tangent measure to $\mu$ is flat 
 (i.e., is a multiple of the restriction of $\mathcal H^d$ to a $d$-plane).
This implies that $\mu$ is $d$-rectifiable and thus (since $\mu$ is Ahlfors regular) 
that $E$ is $d$-rectifiable.

Let $Q\in E$ be a point such that the non-tangential limit of $|\nabla D_{\mu, \beta}|$ 
exists for every aperture and let 
 $\{ r_i \}$ be any sequence of positive numbers that tends to $0$.
Then define $E_i, \mu_i$ and $D_{i} \equiv D_{\mu_i,\beta}$ as in \eqref{e:blowupsforrect}. 
Lemma \ref{l:convunderblowups} shows that, passing to subsequence 
if needed, we may assume that $E_i$ tends to a limit $E_\infty$,
$\mu_i$ has a weak limit $\mu_\infty$, and $D_{i}$ converges, uniformly on compact subsets
of $\R^n \sm E_\infty$, to $D_{\infty} \equiv D_{\mu_\infty, \beta}$.

We now want to show that 
 $|\nabla D_{\infty}|$ 
is constant on 
 $\Omega_\infty \equiv \mathbb R^n \backslash E_\infty$ 
and is equal to 
 $L = \ntlim_{x \rightarrow Q} |\nabla D_{\mu, \beta}(x)|$. 
Let 
$Y, Z \in \Omega_\infty$ and set $\eta_Y = \mathrm{dist}(Y, E_\infty)/(2|Y|) \in (0,1)$, and similarly let $\eta_Z = \mathrm{dist}(Z, E_\infty)/(2|Z|) \in (0,1)$.  We can assume that $\eta_Z \leq \eta_Y$ so that $\Gamma_{1, \eta_Y}(Q) \subseteq \Gamma_{1, \eta_Z}(Q)$.
By the convergence of $E_i$ to $E_\infty$,
we have $Z \in \Omega_i$ for $i$ large enough and $\mathrm{dist}(Z, E_i) \geq \mathrm{dist}(Z, E_\infty)/2$. Therefore, $r_iZ + Q\in \Omega$ and $\mathrm{dist}(r_iZ + Q, E) = r_i \mathrm{dist}(Z, E_i) \geq r_i\mathrm{dist}(Z, E_\infty)/2 = r_i \eta_Z|Z|$. 
Thus $r_iZ + Q \in \Gamma_{1, \eta_Z}(Q)$ for $i$ large enough (where the cone is with respect to $\Omega$). 
Similarly $r_i Y + Q \in \Gamma_{1, \eta_Y}(Q)$ for $i$ large enough (again where the cone with with respect to $\Omega$). 
Observe that $L = \ntlim_{x \rightarrow Q}^{\eta_Z} |\nabla D_{\mu, \beta}(x)|$ because the 
nontangential convergence holds in every cone.
We can then write 
$$
|\nabla D_{\mu_\infty, \beta}(Z)| =  \lim_{i \to +\infty} 
|\nabla D_{\mu, \beta}(r_i Z + Q)| 
=L= \lim_{i} |\nabla D_{\mu, \beta}(r_iY + Q)| = |\nabla D_{\mu_\infty, \beta}(Y)|,
$$
where the first and last equalities follow from \eqref{e:blowupsforrect}, $\mu_i \rightharpoonup \mu_\infty$ and Lemma \ref{l:convunderblowups}.

We conclude that $|\nabla D_{\mu_\infty, \beta}|$ is constant on $\Omega_\infty$. 
If that constant is zero then by the fact that $D_{\mu_\infty, \beta}$ vanishes on 
$E_\infty$ we get that $D_{\mu_\infty, \beta} \equiv 0$. This
 contradicts \eqref{1.2}-\eqref{1.3}.
Thus, $|\nabla D_{\mu_\infty, \beta}|$ is a non-zero constant on $\Omega_\infty$ and 
 by Corollary \ref{t2.2},
$E_\infty$ is a $d$-dimensional affine space and 
$\mu_\infty$ is a constant multiple of $\mathcal H^d$ restricted to $E_\infty$. Thus $\mu_\infty$ is flat. 

In the language of tangent measures, all the tangents to $\mu$ at $Q$ are flat measures. 
Furthermore by Ahlfors regularity the upper density of $\mu$ is bounded away from infinity 
and the lower density of $\mu$ is bounded away from zero. 
Thus we can invoke Theorem 17.6 in \cite{Mattila}
and conclude that the support of $\mu$ is a $d$-rectifiable set. 
Since $E$ is the support of $\mu$ we are done. 
\end{proof}

Finally, we can compute the non-tangential limit of $|\nabla D_{\mu, \beta}|$ 
at a point at which the $d$-density of $\mu$ exists  (call it $\Theta^d(\mu, Q)$)
and $E$ has a unique tangent plane, call it $T_QE$. Blowing up at such a point gives $\mu_\infty = \Theta^d(\mu, Q)\mathcal H^d|_{T_QE}$. Recalling that $D_{\mathcal H^d|_P, \beta} = c_{\beta, n,d} \delta_P$ for any plane $P$, we have
 \begin{equation}\label{e:dinfinity} D_{\infty} = D_{\Theta^d(\mu, Q)\mathcal H^d|_{T_QE}, \beta} = c_{\beta, n, d} \Theta^d(\mu, Q)^{-1/\beta} \delta_{T_QE},\end{equation} which implies that \begin{equation}\label{e:ntlimitatnicepoint} \ntlim_{x\rightarrow Q} |\nabla D_{\mu, \beta}| = c_{\beta, n, d} \Theta^d(\mu, Q)^{-1/\beta}.\end{equation}

\section{$D_\alpha$ 
for ``magic-$\alpha$"}
\label{s:MagicAlpha}

Let $E \subset \mathbb R^n$ be a $d$-Ahlfors regular set and let $\mu$ be a $d$-Ahlfors regular measure supported on $E$.  If the numbers $n$, $d < n$ (not necessarily integer), and $\alpha > 0$ are such that
\begin{equation}\label{e:magicalpha}
n = d+2+\alpha
\end{equation}
then it turns out that the function $D_{\mu, \alpha}$ defined in \eqref{1.3} is a solution of 
$$
L_{\mu, \alpha} u \equiv -\mathrm{div}\left(\frac{1}{D_{\mu, \alpha}^{n-d-1}}\nabla u\right)= 0
$$ 
 in $\Omega = \R^n \backslash E$ (throughout this section $\alpha$ will satisfy \eqref{e:magicalpha} whereas $\beta > 0$ will be arbitrary. In particular, we will try to make it clear when we are assuming that $d < n-2$). 

We can check this (in the classical sense) in $\Omega$ simply by differentiating the smooth function $D_{\mu, \alpha}$ (recall \eqref{e:diffD}),
\begin{eqnarray}\label{5.2}
L_{\mu, \alpha} D_{\mu, \alpha} &=& -{\rm div} \left(D_{\mu, \alpha}^{-n+d+1} \nabla D_{\mu, \alpha}\right) = 
\frac{1}{\alpha} {\rm div}\left( D_{\mu, \alpha}^{-n+d+1} R_{\mu, \alpha}^{-\frac{1}{\alpha}-1}\nabla R_{\mu, \alpha}\right)
\nn\\
&=&  \frac{1}{\alpha} {\rm div}\left( D_{\mu, \alpha}^{-n+d+1} D_{\mu, \alpha}^{1+\alpha} \nabla R_{\mu, \alpha}\right) = \frac{1}{\alpha} \Delta R_{\mu, \alpha}
\end{eqnarray}
by \eqref{1.3} and \eqref{e:magicalpha}. Then by \eqref{1.1}, (and \eqref{e:magicalpha})
\begin{equation}\label{5.3}
R_{\mu, \alpha}(x) \equiv \int_{y\in E} |x-y|^{-d-\alpha} d\mu(y) = \int_{y\in E} |x-y|^{2-n} d\mu(y);
\end{equation}
we recognize the Green kernel (notice that $n>2$ by \eqref{e:magicalpha}); hence $L_{\mu, \alpha}D_{\mu, \alpha}=0$ on $\Omega$.

We want to say that $D_{\mu, \alpha}$ is the ``Green function with pole at infinity" associated to the operator $L \equiv L_{\mu, \alpha}$ (indeed, it is a solution which behaves like distance to the boundary). To do so properly however, we need to define the Green function with a pole at infinity (and the corresponding harmonic measure). We will then show that in the complement of any $d$-Ahlfors regular set, $E$, these objects exist and are unique up to multiplication by a positive scalar. Throughout, we will use some of the elliptic regularity and potential theory studied in \cite{DFM2}, in particular we will assume the reader is comfortable with the existence and properties of a Green function and associated harmonic measure with finite pole.

Before we begin we must recall the weighted Sobolev spaces introduced in \cite{DFM2}. Throughout this section $E$ will be a closed $d$-Ahlfors regular set and $\delta(x)$ will denote the distance from $x$ to the closest point in $E$. 

\begin{defi}\label{d:weightedsobolev}[see \cite{DFM2}]
Let $E \subset \mathbb R^n$ be a $d$-Ahlfors regular set for some $d < n-1$ (not necessarily an integer). Set $w(x) \equiv \delta(x)^{-(n-d-1)}$ and define the weighted Sobolev space 
$$
W\equiv W_w^{1,2} \equiv \{u \in L^1_{\mathrm{loc}}(\R^n \backslash E): \; 
\nabla u \in L^2(\mathbb R^n\backslash E, wdx)\}.
$$ 
We can then localize these Sobolev spaces: for any open $\mathcal O \subset \mathbb R^n$ we define 
$$
W_r(\mathcal O) = \{u \in L^1_{\mathrm{loc}}(\mathcal O), \; 
\varphi f \in W
\: \text{for all}\:\: \varphi \in C_0^\infty(\mathcal O)\}.
$$
\end{defi}

It will be useful later to know that $w(x)$ is locally integrable. Indeed, it follows from Ahlfors regularity that $$ |\{x\in B(Q,r) \mid w(x) > \lambda\}| = |\{x\in B(Q,R)\mid \delta(x) < \lambda^{-\frac{1}{n-d-1}}\}| \leq C\lambda^{-\frac{n-d}{n-d-1}}R^d,$$ for all $R > 0$ and $Q\in E$, which in turn implies that $$\begin{aligned} \int_{B(Q,R)} w(x)dx \leq& \int_0^{R^{-(n-d-1)}} |B(Q,R)|d\lambda + CR^d\int_{R^{-(n-d-1)}}^\infty \lambda^{-\frac{n-d}{n-d-1}} d\lambda\\ =& CR^n R^{-(n-d-1)}+ CR^d[-(n-d-1)\lambda^{-\frac{1}{n-d-1}}]_{R^{-(n-d-1)}}^\infty \leq CR^{d+1}.\end{aligned}$$ (We thank a referee for pointing a minor error in the previous version of this computation and for providing us with a fix). 

These Sobolev spaces are the setting in which the elliptic estimates and potential theory established in \cite{DFM2} hold. We can now define the Green function and harmonic measure with a pole at infinity. 

\begin{defi}\label{d:poleatinfinity}
Let $E \subset \mathbb R^n$ be a $d$-Ahlfors regular set for some $d < n-1$ (not necessarily an integer), $\beta \in (0,1)$ and $\mu$ be a $d$-Ahlfors regular measure supported on $E$. Let $\Omega = \mathbb R^n \backslash E$, $D\equiv D_{\mu, \beta}$ be as in \eqref{1.3} and $L \equiv L_{\mu,\beta}$ be the associated degenerate elliptic operator. We say that $u_\infty, \omega_\infty$ are the Green function and harmonic measure with pole at infinity, respectively (associated to $\beta, \mu$), if $u_\infty \in W_r(B(Q, R))\cap C(\mathbb R^n)$ for every $Q\in E$ and $R > 0$ and the following holds:
\begin{equation}\label{e:atinfinity}\begin{aligned}
u_\infty > \,& 0,\: \text{in}\: \Omega, \\
u_\infty = \,& 0,\: \text{on}\: E,\\
L u_\infty =\, &0,\: \text{in}\: \Omega,\\
\int_{\Omega}D^{-(n-d-1)} \nabla  u_\infty \cdot \nabla \varphi \, dX =& \int_E \varphi d\omega_\infty,\: \forall \varphi \in C^\infty_0(\mathbb R^n).
\end{aligned}
\end{equation}
\end{defi}

Before we can show the existence and uniqueness of these objects, we must recall the comparison principle for solutions, stated and proven in our setting in \cite{DFM2} (see also, e.g.,  \cite{jerisonkenigNTA} for the co-dimension one statement). Recall from \cite{DFM2} that there exists an $M > 1$ such that for $Q \in E$ and $r > 0$ there exists a point $A_r(Q)$ with \begin{equation}\label{e:corkscrew} |A_r(Q) - Q| \leq r \leq M\delta(A_r(Q)).\end{equation} We call $A_r(Q)$ a {\bf corkscrew point} for $Q$ at scale $r >0$. 

\begin{thm}\label{t:boundarycomparison}[Theorem 11.146 in \cite{DFM2}]
Let $Q \in E$ and $r > 0$ and let $X_0 = A_r(Q) \in \Omega$ be the corkscrew point 
for $Q$ at scale $r$. Let $u, v \in W_r(B(Q, 2r))$ be non-negative, not identically zero, 
solutions of $L_{\mu, \beta}u = L_{\mu, \beta} v = 0$ in $B(Q, 2r)$, $\beta>0$, such that $Tu = Tv = 0$ 
on $E \cap B(Q, 2r)$ (where $T$ is the trace operator defined in Theorem 3.4 in \cite{DFM2}). Then there exists a constant $C > 1$ depending on $n,d$, and Ahlfors regularity constants, such that \begin{equation}\label{e:boundarycomparison} 
C^{-1}\frac{u(X_0)}{v(X_0)} \leq \frac{u(X)}{v(X)} \leq C\frac{u(X_0)}{v(X_0)}, 
\forall X \in \Omega\cap B(Q, r).
\end{equation}
\end{thm}

The comparison theorem leads naturally to H\"older regularity of quotients at the boundary. Our proof below is inspired by \cite{FSY} (Theorem 4.5 there), who show this regularity for solutions of a parabolic problem. The ``usual" elliptic proof (cf. \cite{jerisonkenigNTA}) relies on interior approximating domains, which are difficult in the co-dimension greater than one setting because of the presence of boundaries with mixed dimension (see the discussion at the beginning of Section 11 of \cite{DFM2}). 

\begin{cor}\label{c:boundaryholdercomparison}
Let $u, v, Q, r$ be as in Theorem \ref{t:boundarycomparison}. There exists 
$c >0, \gamma \in (0,1)$ (depending only on the Ahlfors regularity of $E$, $n$, $d$ and $\beta$) such that \begin{equation}\label{e:boundaryholdercomparison} \left|\frac{u(X)v(Y)}{u(Y)v(X)} - 1\right| \leq c \left(\frac{\rho}{r}\right)^\gamma, \end{equation} for all $X,Y\in B(Q,\rho)\cap \Omega$, as long as $\rho < r/4$. 
\end{cor}

\begin{proof}
We claim there exists some $\theta \in (0,1)$ (independent of $Q$ and $r$) such that \begin{equation}\label{e:improvementofoscillation}\mathrm{osc}_{B(Q,r/2)} \frac{u}{v} \leq \theta \,\mathrm{osc}_{B(Q,r)} \frac{u}{v},\end{equation} for all $r < r/2$. That the claim implies \eqref{e:boundaryholdercomparison} follows from iterating  \eqref{e:improvementofoscillation} and appealing to Theorem \ref{t:boundarycomparison}. 

Let $\inf_{B(Q,r)} \frac{u}{v} = c_1$ and $\sup_{B(Q,r)} \frac{u}{v} = c_2$ and replace 
$u$ by $U = \frac{u-c_1v}{c_2-c_1}$ (if $c_2 = c_1$ then $u = c_1v$ and the result is trivial). 
It follows that $L_{\mu,\beta} U = 0$ and $U \geq 0$ in $B(Q,r)$ and $U = 0$ on 
$E \cap B(Q,r)$. So we can apply Theorem \ref{t:boundarycomparison} with $2r$ replaced by $r$. Also note that 
$$
0 \leq \frac{U}{v}(Z) \leq 1 = \mathrm{osc}_{B(Q,r)} \frac{U}{v}, \: \forall Z \in B(Q,r)\cap \Omega.
$$ 
Note that $\mathrm{osc}(U/v) = (c_2 -c_1)^{-1}\mathrm{osc}(u/v)$. So if estimate  \eqref{e:improvementofoscillation} holds for $U, v$ it also holds for $u,v$. 

Let $A_{r/2}(Q)$ be the corkscrew point for $Q$ at scale $r/2$. If $\frac{U}{v}(A_{r/2}(Q))< C^{-2}$ (where $C > 1$ is the constant from Theorem \ref{t:boundarycomparison}), then by Theorem \ref{t:boundarycomparison} we would have $$0 \leq \frac{U}{v}(Z) \leq C \frac{U}{v}(A_{r/2}(Q))< \frac{1}{C},\: \forall Z \in B(Q, r/2)\cap \Omega,$$ which would imply that $\mathrm{osc}_{B(Q,r/2)}\frac{U}{v} < \frac{1}{C}$, and hence, the desired result (with $\theta = \frac{1}{C}$). 

If, on the other hand, $\frac{U}{v}(A_{r/2}(Q)) > C^{-2}$, we apply Theorem \ref{t:boundarycomparison} to obtain $$\inf_{B(Q,r/2)}\frac{U}{v} \geq C^{-1} \frac{U}{v}(A_{r/2}(Q)) > C^{-3}.$$

This implies that $\mathrm{osc}_{B(Q,r/2)} \frac{U}{v} \leq 1- C^{-3}$ and so \eqref{e:improvementofoscillation} holds with $\theta = 1-C^{-3}$. 
\end{proof}

Finally, using an argument inspired by \cite{kenigtoroannals}, Lemma 3.7 and Corollary 3.2, we can show the existence and uniqueness of the Green function and harmonic measure with pole at infinity. 

\begin{lem}\label{l:existenceatinfinity}
For any $E, \beta, \mu$ as in Definition \ref{d:poleatinfinity}, there exist an associated harmonic measure and Green function with pole at infinity. Furthermore, they are both unique up to multiplication by a positive scalar. 
\end{lem}

\begin{proof}
First we show existence of the Green function with pole at infinity. 
 Fix 
$Q \in E$ and let $X_i = A_{2^i}(Q) \in \Omega$  denote a corkscrew point for $Q$ at the scale $2^{i}$ (i.e., 
$M\delta(X_i) > |X_i - Q| \geq 2^i$,  see \eqref{e:corkscrew}).
Define (for $i > 1$) $g_i(X) \equiv \frac{g(X, X_i)}{g(X_1, X_i)}$, where 
$g(X,Y)$ is the Green function for $L_{\beta, \mu}$ with pole at $Y$ 
(cf. Section 10 in \cite{DFM2}). 
Similarly define $\omega_i(S) \equiv \frac{\omega^{X_i}(S)}{g(X_1, X_i)}$. 
These are somewhat arbitrary normalizations (recall that the Green function that we want
to construct will only be unique modulo a multiplicative function).

We  claim that for any $K \subset \subset \mathbb R^n$ there exists a $C > 0$ (depending on $K$) 
such that $g_i(X) < C$ for all $X \in K$ and $i > i_0(K) \geq 0$ large enough (so that $X_i$ lies away from $K$).  Indeed, this follows from the fact that $g_i(X_1) \equiv 1$, the Harnack inequality (Lemma 8.42 in \cite{DFM2}), and the existence of Harnack chains in $\Omega$ (Lemma 2.1 in \cite{DFM2}). From this it follows that the $g_i$ are uniformly H\"older continuous on compacta (see Lemma 8.42 and Lemma 8.98 in \cite{DFM2}). Thus (after the extraction of a diagonal subsequence) we have that $g_i \rightarrow g_\infty$ where the convergence is uniform on compacta in the continuous topology. Note that in $K \subset \subset \Omega$ the equation is uniformly elliptic, so the uniform convergence on compacta also implies (again perhaps passing to a subsequence) that $\nabla g_i \rightarrow \nabla g_\infty$ pointwise almost everywhere. Finally, note that the uniform convergence implies that $g_\infty \geq 0$ in $\Omega$ and $g_\infty = 0$ on $E$. Furthermore, by the Harnack inequality and $g_\infty(X_1) = 1$ it must be that $g_\infty > 0$ in $\Omega$.

For any $Q \in E$ and $R > 0$, if $i > i_0(Q,r) \geq 0$ is large enough we know 
that $X_i \notin B(Q, 4R)$. Thus we can estimate 
$$
\int_{B(Q,R)}|\nabla_x g_i(x)|^2 w(x)dx \leq CR^{-2} 
 \int_{B(Q,2R)} 
g_i(x)^2 w(x)dx \leq C_R,
$$ 
where the first  inequality follows from Lemma 8.47 in \cite{DFM2} (a Caccioppoli type estimate) and the second inequality follows from the fact that $|g_i| < C_R$ on $B(Q, 2R)$ by the argument in the above paragraph (and the fact that $w(x)$ is locally integrable). 
Thus the $g_i$ are in $W_r(B(Q,R))$ with uniformly controlled norms for all $i$ large enough and, 
applying Fatou's lemma, we conclude that $g_\infty$ is in $W_r(B(Q,R))$ for all $Q \in E$ and $R > 0$. 

As the $g_i$ are in $W_r(B(Q,R))$ with uniformly controlled norms, it follows from the weak formulation of  
$L_{\mu, \beta}g_i = 0$ (and integration by parts) that $L_{\mu, \beta}g_\infty = 0$ in $\Omega$. 

We will now show that $g_\infty$ is the unique positive solution to $L_{\mu, \beta}$ which vanishes on $E$ and is in $W_r(B(Q,R))$ for all $Q\in E$ and $R > 0$ (up to scalar multiplication). Indeed, assume there existed some other $f$ which was positive in $\Omega$, zero on $E$, in $W_r(B(Q,R))$ for all $R > 0, Q\in E$ and satisfied $L_{\mu, \beta} f = 0$. We can multiply $f$ by a positive scalar such that $f(X_1) = 1$. Then by Corollary \ref{c:boundaryholdercomparison} applied at larger and larger scales it is clear that $f(X) = g_\infty(X)$ for all $X \in \Omega$: starting from 
$$\left|\frac{g_\infty(X)}{f(X)}-1\right| \leq C \left(\frac{|X-X_1|}{R}\right)^\gamma,$$ 
take $R\to\infty$.

It is time to establish the existence of $\omega_\infty$; let $Q \in E$ and $R > 0$. If $i$ is big enough then by Lemma 11.78 in \cite{DFM2} we have $\omega^{X_i}(B(Q, R)) \leq C R^{d-1}g(X_i, A_R(Q))$ where $A_R(Q)$ is the corkscrew point for $Q$ at scale $R$. Thus $\omega_i(B(Q,R)) \leq C R^{d-1} g_i(A_R(Q)) \leq C_{R,Q} < \infty$ by the Harnack argument above. This implies that the sequence of measures $\{\omega_i\}$ is uniformly bounded on any compact set and so, perhaps passing to a subsequence, there is an $\omega_\infty$ such that $\omega_i \rightharpoonup \omega_\infty$. 

We note that by the definition of $g_i$ and $\omega_i$ we have that $$\int_{\Omega} D_{\mu, \beta}^{-(n-d-1)} \nabla g_i \nabla \varphi  dx = \int_E \varphi d\omega_i,$$ for all $\varphi \in C_0^\infty(\mathbb R^n \backslash X_i)$ (cf. Section 9 in \cite{DFM2}). Fix $\varphi$ and let $i \rightarrow \infty$ on both sides; using the fact that $\omega_i \rightharpoonup \omega_\infty$ and $g_i \rightarrow g_\infty$ in $W_r(K)$ for any compact $K$ we get that $$\int_{\Omega} D_{\mu, \beta}^{-(n-d-1)} \nabla g_\infty \nabla \varphi  dx = \int_E \varphi d\omega_\infty,$$ as desired. 

The uniqueness of $\omega_\infty$ then follows from its integral relationship with $g_\infty$ (cf. \eqref{e:atinfinity}) and the uniqueness of $g_\infty$. 
\end{proof}

We shall now show that for the magic value of $\alpha = n - d -2$, $D_\alpha$ is the Green function with pole at infinity.

\begin{cor}\label{c:gfatinfinity}
Let $E \subset \mathbb R^n$ be a $d$-Ahlfors regular set for $d < n-2$ (not necessarily an integer) and let $\mu$ be a $d$-Ahlfors regular measure supported on $E$. If $\alpha = n-d-2$ then $D_{\mu, \alpha}$ is the Green function with pole at infinity for $E$ (cf. Definition \ref{d:poleatinfinity}). 
\end{cor}
We remark that the Green function with pole at infinity is unique modulo a multiplicative constant, hence, strictly speaking, the Corollary above assures that any such Green function is either $D_{\mu, \alpha}$ or its multiple. 

\begin{proof} 
We remark that the uniqueness in Lemma \ref{l:existenceatinfinity} does not require that $D_{\mu,\alpha}$ verify the last condition in \eqref{e:atinfinity}; the proof shows that a non-negative solution to the degenerate operator which vanishes at the boundary and satisfies the correct growth condition is unique up a scalar multiple. 

We have seen earlier that $D_{\mu, \alpha}$ is a positive solution to the degenerate elliptic operator, which vanishes on the boundary; 
because of Lemma \ref{l:existenceatinfinity},
it suffices to show that $D_{\mu, \alpha} \in W_r(B(Q,R))$ for all $Q \in E$ and $R > 0$ 
However, we know from \eqref{e:boundongradD} that $|\nabla D_{\mu, \alpha}| \leq C$. Since $w(X)$ is locally integrable in $\mathbb R^n$, the desired result follows. 
\end{proof}

The fact that we are able to explicitly write down the Green function with pole at infinity for 
magic $\alpha$ allows us to easily compute and bound the associated harmonic measure. 
 The next theorem
shows that, for any Ahlfors regular set $E$ and magic $\alpha = n-d-2$, the harmonic measure 
$\omega_{\mu, \alpha}$ is comparable to surface measure. As a corollary we have the analogous result for harmonic measure with finite pole. Thus, as mentioned in the introduction, there is absolutely no converse to the theorem that $\omega^X \in A_\infty(\sigma)$ when $E$ is uniformly rectifiable. In fact, our result holds even when $d < n-2$ is not an integer. 

\begin{thm}\label{t:magicomegaatinfinity}
Let $n\geq 3$, $E \subset \mathbb R^n$ be a $d$-Ahlfors regular set (for $d < n-2$, not necessarily an integer) 
and let $\mu$ be a $d$-Ahlfors regular measure supported on $E$. 
If $\alpha = n -d-2$ then the harmonic measure with pole at infinity $\omega_{\mu, \alpha}$ 
is comparable to $\sigma \equiv \mathcal H^d|_E$; that is, there is a constant $C > 0$ 
(depending only on $n, d$ and the Ahlfors regularity  constants for $\mu$ and $\sigma$)
such that if we normalize $\omega_{\mu, \alpha}$  as we did in the construction
\begin{equation}\label{e:boundedaboveandbelow}
C^{-1}\leq  \frac{d\omega_{\mu, \alpha}}{d\sigma}(Q) \leq C,\: \forall Q \in E.
\end{equation} 
\end{thm}

For the rest of the section we will use the notation $a\simeq b$ if there is a constant $C$, depending only on $n,d$ and the Ahlfors regularity of $\mu$, such that $C^{-1} \leq \frac{a}{b} \leq C$. 

\begin{proof}
For the sake of brevity we will write $\omega = \omega_{\mu, \alpha}, D = D_{\mu, \alpha}$. 
Recall Lemma 11.78 from \cite{DFM2}: there exists a $C> 0$ 
(depending on $n, d$, and the Ahlfors regularity constants of $E, \mu$) such that 
if $Q \in E, r > 0$ and $X \in \Omega \setminus B(Q, 2r)$ then 
\begin{equation}\label{e:finitepolecomparison} 
C^{-1} r^{d-1} g(X, X_0) \leq \omega^X(B(Q,r)) \leq Cr^{d-1}g(X, X_0),
\end{equation}
where $X_0 = A_r(Q)$ is a corkscrew point for $Q$ at scale $r$, $g(-,X_0)$ is the Green function with pole at $X_0$ associated to $L_{\mu, \alpha}$ and $\omega^X$ is the harmonic measure with pole at $X$ associated to $L_{\mu,\alpha}$. 
Divide \eqref{e:finitepolecomparison} by $g(X_i,X_1)$ for $i > 1$, 
 where $X_i = A_{4^ir}(Q)$, and take $X = X_i$. This yields
$$
C^{-1} r^{d-1} \frac{g(X_i, X_0)}{g(X_i, X_1)} 
\leq \frac{\omega^{X_i}(B(Q,r))}{g(X_i, X_1)} \leq Cr^{d-1}\frac{g(X_i, X_0)}{g(X_i, X_1)}.
$$
Then we let $i$ tend to $+\infty$; we claim that 
\begin{equation} \label{e:infinitepolecomparison}
C^{-1}r^{d-1}D(X_0) \leq \omega(B(Q,r)) \leq Cr^{d-1}D(X_0).
\end{equation}
Indeed, arguing as in Lemma \ref{l:existenceatinfinity}, the Harnack inequality implies that $G_i(-) \equiv \frac{g(X_i, -)}{g(X_i, X_1)}$ are uniformly, in $i$, bounded on compacta, are all positive and harmonic in $\mathbb R^n \backslash (E \cup \{X_i\})$ and zero on $E$. Passing to a (subsequential) limit we get that $G_i(-)\rightarrow G_\infty$ a function which satisfies the definition of a Greens function at infinity. Using the uniqueness of said function we can conclude that $G_\infty = C D$ (in fact we compute that $C = D(X_1)$). Similarly, the measures $\omega_i \equiv \frac{\omega^{X_i}}{g(X_i, X_1)}$ form a pre-compact sequence in the weak-topology and, with the $G_i$, satisfy the last line of \eqref{e:atinfinity}. This equation is preserved under the uniform convergence of the $G_i$ and the weak-limit of the $\omega_i$ and, as such, $\omega_i \rightharpoonup \omega_\infty$, the harmonic measure with pole at infinity. With this convergence in mind, the inequality \eqref{e:infinitepolecomparison} follows from the prior offset inequality letting $i \rightarrow \infty$ (one also has to use the doubling of $\omega$ to see that $\omega(\overline{B(Q,r)}) \leq C\omega(B(Q,r))$). 

From \eqref{e:infinitepolecomparison} the conclusion of the Lemma is easy:  notice 
that $D(X_0) \simeq \delta(X_0) \simeq r$. It follows that $\omega(B(Q,r)) \simeq r^d$ 
for any $Q\in E$ and any $r > 0$. 
\end{proof}

There is an analogue to Theorem \ref{t:magicomegaatinfinity} for a 
harmonic measure with finite pole associated to $L_{\mu, \alpha}$ (where $\alpha$ is magic).  The proof essentially follows from the boundedness of quotients and the comparability of the Green function with the harmonic measure (Theorem 11.146 and Lemma 11.78 in \cite{DFM2} respectively).

\begin{cor} \label{c:magicomegafinite}
Let $n\geq 3$, $E \subset \R^n$ be an Ahlfors regular set of dimension $d$ (not necessarily integer), 
and $\mu$ be an Ahlfors regular measure with support $E$. Assume that $\alpha=n-d-2>0$, 
and define $D = D_{\mu, \alpha}$ and $L=L_{\mu, \alpha}$ as above. 
Finally let $X \in \Omega = \R^n \sm E$ be given, and denote by
$\omega^X$ the associated harmonic measure with pole at $X$. Set $R = \dist(X,E)$
Then there is a constant $C$,  that depends only on $n$, $d$, and the Ahlfors regularity constant 
for $\mu$, such that 
\begin{equation} \label{e:finitepolecomparability}
C^{-1} \mu(A) \leq R^d \omega^X(A) \leq  C \mu(A)
\ \text{ for every measurable set } A \subset E \cap B(X,100R).
\end{equation}
\end{cor}
We remark that \eqref{e:finitepolecomparability} is a correct, homogeneous, finite pole version of the statement that the harmonic measure is proportional to the Hausdorff measure. A reader might be more accustomed to see it as a strengthened version of the $A^\infty$ condition: for all $Q\in E$, $X\in \Omega$, $R = \dist(X,E)$, 
\begin{equation} \label{e:finitepolecomparability-bis}
C^{-1} \frac{\mu(A)}{\mu(B(X,100R))} \leq \frac{\omega^X(A)}{\omega^X(B(X,100R))} \leq  C  \frac{\mu(A)}{\mu(B(X,100R))},
\end{equation}
for every measurable set $A \subset E \cap B(X,100R)$.

\begin{proof}
Of course there is nothing special about the radius $100R$, but the result for larger $R$ 
could be obtained by a change of pole.

 Rather than proving \eqref{e:finitepolecomparability}, we will find it more convenient to prove
 that for
all $Q \in E, r> 0$ such that $B(Q,r)\subset B(X, 100R)$ and $r < R/4$, 
\begin{equation} \label{e:finitepolecomparability2}
C^{-1} r^d \leq R^d \omega^X(B(Q,r)) \leq  C r^d.
\end{equation}
By Lemma 11.78 in \cite{DFM2} we know that $\omega^X(B(Q,r))\simeq r^{d-1}g(X,A_r(Q))$ where $g$ is the Green function associated to the operator $L_{\alpha,\mu}$. Note that $g(X,-)$ and $D(-)$ are both positive solutions to $L$ in $B(Q, R/2)$ which vanish on $B(Q,R/2)\cap E$, thus we can apply Theorem \ref{t:boundarycomparison} and get that 
\begin{equation}\label{e:chainofcomparability} 
\frac{\omega^X(B(Q,r))}{r^d} 
\simeq \frac{g(X, A_r(Q))}{r} 
\simeq \frac{g(X, A_r(Q))} {D(A_r(Q))} 
\simeq \frac{g(X, A_{R/2}(Q))}{D(A_{R/2}(Q))} 
\simeq \frac{g(X, A_{R/2}(Q))}{R},
\end{equation} 
 so \eqref{e:finitepolecomparability2} will follow once we check that 
$g(X, A_{R/2}(Q)) \simeq R^{1-d}$.

Choose $Y \in \Omega$ so that $R/20< |Y-X| < R/10$ and consider $g(X, Y)$. 
It is clear that $\dist(Y, E) \simeq R \simeq \dist(A_{R/2}(Q), E) \simeq |A_{R/2}(Q)- Y|$. 
Thus by the existence of Harnack chains (Lemma 2.1 in \cite{DFM2}), 
 because we can find a chain from $Y$ to $A_{R/2}(Q)$ that does not get close to $X$, 
and by the Harnack inequality,  we have $g(X, A_{R/2}(Q)) \simeq g(X, Y)$. 

Finally, by equations (10.89) and (10.96) in \cite{DFM2} we have that $g(X,Y) \simeq R^{1-d}$. 
 Plugging this into \eqref{e:chainofcomparability} gives 
\eqref{e:finitepolecomparability2}; then \eqref{e:finitepolecomparability} follows 
from the Lesbesgue differentiation theorem (or, equivalently, by a straightforward covering argument).
\end{proof}

When $E$ is rectifiable the non-tangential limit of $|\nabla D_{\mu,\alpha}|$ exists 
(Theorem \ref{thm:rectimpliesntlimits}), and, much as in the co-dimension one setting, 
gives us the Poisson kernel (for magic $\alpha$).

\begin{lem}\label{l:poissonkernelisNTlimit}
Let $n \geq 3$ and let $E \subset \mathbb R^n$ be a $d$-Ahlfors regular set with $d < n-2$ an integer. 
Assume that $E$ is $d$-rectifiable, let $\mu$ be a $d$-Ahlfors regular measure whose support 
is $E$ and let $\alpha = n -d- 2$. Then 
for $\sigma$-a.e. $Q\in E$, the density of $\omega_{\mu, \alpha}$ is given by $\Theta^d(\mu, Q)$ modulo a multiplicative constant. To be precise, if we fix the constants so that the Green function with pole at infinity is the function $D_{\mu, \alpha}$ that was constructed above, 
then there exist  $c_{n, d}>0$ and 
$\widetilde{c}_{n,d}>0$ 
such that 
$$
\frac{d\omega_{\mu, \alpha}}{d\sigma}(Q) 
= \widetilde{c}_{n,d}  \ntlim_{x \to Q} |\nabla D_{\mu, \alpha}(x)|^{-(n-d-2)}
\stackrel{\eqref{e:ntlimitatnicepoint}}{=}c_{n, d}\Theta^d(\mu, Q).
$$
\end{lem}

\begin{proof}
For simplicity 
 write 
$D \equiv D_{\mu, \alpha}$ and $\omega$ the associated harmonic measure with pole at infinity. For any $\varphi \in C^\infty_c$ we know that 
\begin{equation} \label{addednumber}
\int_{\Omega} D^{-(n-d-1)}  \nabla D \cdot \nabla \varphi \, dx 
= \int_E \varphi d\omega.
\end{equation}
Let $Q\in E$ be a point where the non-tangential limit of $|\nabla D|$ exists and where 
there is a unique tangent to $E$ and tangent measure for $\mu$ (call it $\mu_\infty$). 
Such a $Q \in E$ can be found $\sigma$-a.e.  (by the theory of rectifiable sets and 
Theorem \ref{thm:rectimpliesntlimits} above). 
Let $\varphi$ be a smooth approximation of $\chi_{B(0,1)}$ and for $r_i \downarrow 0$ 
define $\varphi_i(x) \equiv \frac{\varphi(\frac{x-Q}{r_i})}{r_i^d}$. 
Adapting notation as in  \eqref{e:blowupsforrect} we get 
$$
\begin{aligned}
\int_{\Omega} D^{-(n-d-1)} \nabla D \cdot \nabla \varphi_i \, dx 
&= \frac{1}{r_i^{d+1}} \int_{\Omega} D^{-(n-d-1)}(x) \nabla D(x) \cdot
(\nabla \varphi)\big(\frac{x-Q}{r_i}\big) dx \cr
&=
\int_{\Omega_i} D_i^{-(n-d-1)}(y)\nabla D_i(y)\cdot  \nabla\varphi(y)dy,
\end{aligned}
$$
by a change of variables $y = \frac{x-Q}{r_i}$.

We now take 
the limit in $i$. Recall 
that $E$ has a unique tangent $d$-plane $T$ at $Q$, 
and that there is a non-tangential limit $L = \ntlim_{x \to Q} |\nabla D|$. 
In addition, by the discussion in Section~\ref{s:NTLimits} (see, in particular, \eqref{e:dinfinity} and \eqref{e:ntlimitatnicepoint})
$D_i$ tends to $D_\infty(x) = L\delta_T(x)$, and this convergence happens 
uniformly up to $T$. The convergence of $\nabla D_i$ to $\nabla D_\infty$
is only uniform on compact sets of $\R^n \sm T$, but close to $T$
the integrals are controlled uniformly because the gradients are bounded, so we get that

\begin{equation}\label{e:ntlimitispoissonkernel} 
\lim_{i \rightarrow \infty} \int_{\Omega} D^{-(n-d-1)} 
\nabla D \cdot \nabla \varphi_i \, dx 
= L^{-(n-d-2)}   
\int_{\R^n \backslash T} (\delta_T)^{-(n-d-1)} \nabla \delta_T \cdot \nabla \varphi dx.
\end{equation} 

Split the integral on the right hand side of \eqref{e:ntlimitispoissonkernel} into two pieces; one 
 on a neighborhood $T_\varepsilon$
of radius $\varepsilon > 0$ around $T$ and the other outside 
 of $T_\varepsilon$.
The integral on $T_\varepsilon$  
goes to zero as $\varepsilon > 0$ goes to zero, 
by the Lipschitz character of $\delta_T$ and $\varphi$ and the local integrability of 
$\delta^{-(n-d-1)}$. For the integral 
on $\R^n \sm T_\varepsilon$
we can integrate by parts. Notice that $\delta_T$ is a distance to the $d$-tangent plane, hence, 
it is a radial function in a space with $n-d$ dimensions, and hence, $\delta_T^{-n+d+2}$ is harmonic. Then
\begin{equation} \label{6.16n}
\begin{aligned}
\int_{\R^n \backslash T_\varepsilon} 
(\delta_T)^{-(n-d-1)} \nabla \delta_T \cdot \nabla \varphi dx
&= \int_{\{x|\dist(x, T) = \varepsilon\}} \varepsilon ^{-(n-d-1)} \varphi d\mathcal H^{n-1}
\\
&= c_1\int_T \int_{\epsilon \mathbb S^{n-d-1}} \varepsilon^{-(n-d-1)} \varphi d\mathcal H^{n-d-1} d\mathcal H^{d} 
\\
&= c_2\int_T \varphi  d \mathcal H^d + O(\varepsilon),
\end{aligned}
\end{equation}
where $c_1$ and $c_2$ are dimensional constants that we shall never need to compute.
We let $\varepsilon$ tend to $0$, return to \eqref{e:ntlimitispoissonkernel}, and get that
\begin{equation} \label{6.17n}
\lim_{i \rightarrow \infty} \int_{\Omega} D^{-(n-d-1)} 
\nabla D \cdot \nabla \varphi_i \, dx 
= c_2 L^{-(n-d-2)}  \int_T \varphi  d \mathcal H^d.
\end{equation}
Now we let $\varphi$ tend to $\chi_{B(0,1)}$ (as BV functions);  
the right-hand side tends to $c_2 L^{-(n-d-2)} V_d$, where $V_d$ is 
the volume of the $d$-dimensional unit ball. For the left-hand side, notice that by \eqref{addednumber},
$$
\int_{\Omega} D^{-(n-d-1)} \nabla D \cdot \nabla \varphi_i dx 
= \int_E \varphi_i d\omega
= \int_E \varphi_i \, \frac{d\omega}{d\sigma} \, d\sigma.
$$
When $i$ tends to $+\infty$ and $Q$ is a point of density for $\frac{d\omega}{d\sigma}$ 
(which is true $\sigma$-a.e.), the quantity 
$\int_E \varphi_i(Z) \big|\frac{d\omega}{d\sigma}(Z)- \frac{d\omega}{d\sigma}(Q)\big| d\sigma(Z)$ tends to $0$; we are left with
$\frac{d\omega}{d\sigma}(Q) \int_E \varphi_i  d\sigma$.
If $Q$ is also a point of density $1$ for $\sigma$ (which is again true $\sigma$-a.e.), 
$\int_E \varphi_i  d\sigma$ tends to $\int_T \varphi d\H^d$. 
Now we let $\varphi$ tend to $\chi_{B(0,1)}$ and get that the left-hand side of
\eqref{6.17n} tends to $c_{3} \frac{d\omega}{d\sigma}(Q)$, where $c_3$ may depend
on how we normalize $\H^d$ with respect to the Lebesgue measure. Thus \eqref{6.17n}
implies that $\frac{d\omega}{d\sigma}(Q) = c_4 L^{-(n-d-2)}$. This is the desired result. 
\end{proof}

In the specific case where $\mu = \sigma \equiv \mathcal H^d|_E$ and $E$ is $d$-rectifiable, \eqref{e:ntlimitatnicepoint} tells us that $\ntlim_{x\rightarrow Q}|\nabla D_{\sigma, \alpha}| = c_{n,d} \Theta^d(\sigma, Q)^{-1/\alpha} = c_{n,d}$ by the fact that $\Theta^d(\sigma, Q) = 1$ for $\sigma$-a.e. $Q$ in any $d$-rectifiable set. Thus we can conclude that $\omega_{\sigma, \alpha}$ is proportional to $\sigma$. 

\begin{cor}\label{c:poissonkernelconstant}
Let $E, d, n,\alpha$ be as in Lemma \ref{l:poissonkernelisNTlimit}. 
Then there exists a constant $c > 0$ such that 
$\omega_{\sigma, \alpha} = c \sigma$, where, as above, $\sigma = \mathcal H^d|_E$. 
\end{cor}

\end{document}